\documentclass[12pt]{iopart}
\usepackage{amsthm,amssymb,amsmath}
\usepackage[active]{srcltx}
\newtheorem{theorem}{Theorem}
\newtheorem{lemma}[theorem]{Lemma}
\newtheorem{corollary}[theorem]{Corollary}
\newtheorem{remark}[theorem]{Remark}
\def\tfrac#1#2{{\textstyle\frac{#1}{#2}}} 
\def\be#1{\begin{equation}\label{#1}}
\def\ee{\end{equation}}

\def\udag{u^\dagger}
\def\usol{{u_*}}
\def\fdel{f^\delta}
\def\norm#1{\hspace*{0.2ex}\|#1\|}
\def\normE#1{\hspace*{0.2ex}\|#1\|_E}
\def\normEd#1{\hspace*{0.2ex}\|#1\|_{E^*}}
\def\normF#1{\hspace*{0.2ex}\|#1\|_F}
\def\normFd#1{\hspace*{0.2ex}\|#1\|_{F^*}}
\def\dupE#1{\langle #1\rangle_{E^*,E}}
\def\dupF#1{\langle #1\rangle_{F^*,F}}
\def\dupZn#1{\langle #1\rangle_{Z_n^*,Z_n}}
\def\N{\mathbb{N}}
\def\R{\mathbb{R}}
\usepackage{color}
\numberwithin{equation}{section}
\numberwithin{theorem}{section}

\begin{document}

\title{Regularization by Discretization in Banach Spaces}
\author{Uno H\"amarik$^1$, Barbara Kaltenbacher$^2$, Urve Kangro$^1$, Elena Resmerita$^2$}
\address{$^1$ Institute of  Mathematics, University of Tartu, Estonia, 
$^2$ Institute of Mathematics, Alpen-Adria-Universit\"at Klagenfurt, Austria}
\ead{uno.hamarik@ut.ee, barbara.kaltenbacher@aau.at, urve.kangro@ut.ee, elena.resmerita@aau.at}

\begin{abstract}

We consider ill-posed linear operator equations with operators acting between  Banach spaces. For solution approximation, the methods of choice here are  projection methods onto finite dimensional subspaces, thus extending existing results from Hilbert space settings. More precisely, general projection methods, the least squares method and the least error method are analyzed. In order to appropriately choose the dimension of the subspace, we consider a priori and a posteriori choices by the discrepancy principle and by the monotone error rule. Analytical considerations and numerical tests are provided for a collocation method applied to a Volterra integral equation in one dimension space. 
\end{abstract}
\pacs{02.30.Zz, 02.30.Rz, 02.60.Cb}
\submitto{\IP}
\maketitle


\section{Introduction}
Consider an ill-posed linear operator equation
\be{Auf}
Au=f
\ee
with $A\in L(E,F)$ mapping between nontrivial Banach spaces $E$ and $F$. 
In practice only noisy data $f^\delta$ will be given. We  assume here that the noise level $\delta$ satisfying
\be{delta}
\normF{\fdel-f}\leq \delta
\ee
is known and consider convergence of regularized solutions to an exact solution $\usol$ of \eqref{Auf} as $\delta$ goes to zero. 

Regularization by projection onto finite dimensional subspaces of $E$ and/or $F$ has been studied in detail e.g., in \cite{EnNe87b,GHK14,GrNe88,HAG02,MS,Ri,VH85} in the Hilbert space setting. Here the dimension of the projection spaces plays the role of a regularization parameter. The error estimates of \cite
{EnNe87b,Natt77,Ri} allow for an a priori choice of this  dimension, in \cite{GHK14,HAG02,MS,VH85} also an a posteriori choice of the dimension is considered. Our aim is to extend these results (or at least part of them) to the general Banach space setting. This is motivated, e.g., by the use of $L^p$ spaces with $p\not=2$ to recover sparse solutions or to model uniform or impulsive noise.
 Also the space $C(\bar{\Omega})$ of continuous functions on some domain $\Omega$ and its dual $\mathcal{M}(\Omega)$ are of particular interest since our setting  allows then to analyze, e.g. collocation of integral equations as a regularization method.
Note that some results regarding regularization by discretization in Banach spaces are known in a general setting (see \cite{Natt77} and \cite{SchRieSch12}) and  about the quadrature formulae method (see \cite{Apartsyn03,BrPrVa96}), the collocation method (see, e.g.,  \cite
{BrPrVa96,Eggerm83,Eggerm84,Natt77}) and the Galerkin method (see \cite{BrPrVa96}).

Let  $E_n\subseteq E$, $Z_n\subseteq F^*$, $n\in\N$, be finite dimensional nontrivial subspaces  which have the role of approximating the spaces $E$ and $F^*$, respectively. For instance, the subspaces can be chosen in the following manner, as it will be emphasized later,
\be{Espace}
\forall n\in\N, \, : \ E_n\subseteq E_{n+1} \mbox{ and } \overline{\bigcup_{n\in\N} E_n} = E,
\ee
\be{Fspace}
\forall n\in\N,\,: \ Z_n\subseteq Z_{n+1} \mbox{ and } \overline{\bigcup_{n\in\N} Z_n} = F^*.
\ee
The general projection method defines a finite dimensional approximation $u_n$ to $\usol$ by 
\be{projmeth}
u_n\in E_n \mbox{ and } \forall z_n\in Z_n \, : \ \dupF{z_n,A u_n}=\dupF{z_n,\fdel}. 
\ee
As in the Hilbert space case, the least squares method
\be{leastsquares}
u_n\in\mbox{argmin}\{ \normF{A \tilde{u}_n-\fdel}\, : \, \tilde{u}_n\in E_n\}
\ee
and the least error method
\be{leasterror}
u_n\in\mbox{argmin}\{ \normE{\tilde{u}} \, : \, \forall z_n\in Z_n \, : \ \dupF{z_n,A \tilde{u}}=\dupF{z_n,\fdel}\}
\ee
 can be recovered to some extent as special cases of \eqref{projmeth}, see Lemmas \ref{characterization_leastsquares}, \ref{characterization_leasterror} below.

A justification of the name ''least error'' method will be provided later (see Theorem \ref{rem_leasterror} in Section \ref{sec_leasterror}).

In the following, $P_n:E\to E_n $ denotes some projection. 
For drawing certain conclusions, this will sometimes be assumed 
to have the following properties:
\be{P_n}
(I-P_n)^2=I-P_n
{
\mbox{ and }\forall u\in X\,,\ \lambda\in\R\, : \|(I-P_n)(\lambda x)\|=|\lambda|\,\|(I-P_n)(x)\|}. 
\ee
As opposed to the Hilbert space setting, $P_n$ is not necessarily linear any more.
\begin{remark}  i) One can use  the metric projection operator 
\[ P_n:E\to E_n \quad P_n (w)=\mbox{argmin}\{\normE{w-w_n}\, : \, w_n\in E_n\},\] 
in case it is single valued (as happens in strictly convex Banach spaces),  but it can also be some differently defined projection operator, 
for an example see Section \ref{sec_applic} below.
Note that the metric projection $P_n$ is obviously homogeneous, idempotent and   does fulfill $(I-P_n)^2=I-P_n$, as one can see in what follows: $(I-P_n)^2(u)=(I-P_n)(u)$ if and only if
\[
u-P_n(u)-P_n(u-P_n(u))=u-P_n(u),
\]
which is equivalent to $P_n(u-P_n(u))=0$. Indeed,
\[
\normE{u-P_n(u)-0}\leq\normE{u-(P_n(u)+v_n)}=\normE{u-P_n(u)-v_n},
\]
for all $v_n\in E_n$, as $P_n(u)+v_n\in E_n$. 

ii) In general, single valued metric projections onto finite dimensional subspaces of a Banach space $X$ are  nonlinear, otherwise $X$ would be linearly isometric to an inner product space, cf., e.g., \cite[p. 210]{Haz90}.
\end{remark}

Let $Q_n$  be the linear operator  defined by 
\[ Q_n: F\to Z_n^* \quad \forall g\in F\,, z_n\in Z_n \, : \ \dupZn{Q_n g,z_n}=\dupF{z_n,g}\]
which allows to write \eqref{projmeth} as
\be{projmethQn}
u_n\in E_n \mbox{ and } Q_nA u_n=Q_n\fdel.
\ee
The norm of $Q_n$ equals one since
\begin{equation}\label{Qn1}
\begin{aligned}
\|Q_n\|=&\sup_{g\in F, \normF{g}=1}\|Q_ng\|_{Z_n^*}
 =\sup_{g\in F, \normF{g}=1, z_n\in Z_n, \normFd{z_n}=1}\dupZn{Q_n g,z_n}\\
 =&\sup_{g\in F, \normF{g}=1, z_n\in Z_n, \normFd{z_n}=1}\dupF{z_n,g}
 =1.
\end{aligned}
\end{equation}
Moreover,  $Q_n'$ will stand for  the metric projection onto the subspace $AE_n$ (or a single valued choice of the metric projection in case it is multivalued), whenever $E_n$ is a linear subspace of $E$, so that \eqref{leastsquares} can be rewritten as
\[ 
Au_n = Q_n'\fdel.
\]
In the Hilbert space setting, the least squares and the least error method can be shown to be special cases of the general projection method \eqref{projmeth} upon appropriate choice of the spaces $E_n$ and $Z_n$, respectively.
This can be extended to the Banach spaces under certain conditions. For this purpose we will make use of duality mappings 
\be{dualitymapping}
\begin{aligned}
&J_q^{F\to F^*}=\partial (\tfrac{1}{q}\normF{\cdot}^q)=\partial \Phi_F\,, \\
&J_{q^*}^{F^*\to F^{**}}=\partial (\tfrac{1}{q^*}\normFd{\cdot}^{q^*})\,,\  q^*=\tfrac{q}{q-1}\,, \\
&J_q^{E\to E^*}=\partial(\tfrac{1}{q}\normE{\cdot}^q)=\partial \Phi_E\,,
\end{aligned}
\ee
cf., e.g., \cite[Chapters I-II]{Cioranescu90}.

Moreover, we will make use of the Bregman distance induced by the functional $\Phi_E=\tfrac{1}{q}\normE{\cdot}^q$, which in case of single valued duality mapping $J_q^{E\to E^*}$ is defined by
\be{Bregmandistance}
D_q(\tilde{u},u)=\tfrac{1}{q}\normE{\tilde{u}}^q-\tfrac{1}{q}\normE{u}^q + \dupE{J_q^{E\to E^*}(u),u-\tilde{u}}\,.
\ee
We will also use the symmetric Bregman distance 
\be{symmBregmandistance}
D_q^{sym}(\tilde{u},u)=D_q(\tilde{u},u)+D_q(u,\tilde{u})=\dupE{J_q^{E\to E^*}(u)-J_q^{E\to E^*}(\tilde{u}),u-\tilde{u}}
\ee
and the identity
\be{idBregdist}
D_{q^*}(J_q^{E\to E^*}(u),J_q^{E\to E^*}(\tilde{u}))
=D_q(\tilde{u},u)
\ee
provided that
\[ \forall u\in E\, : \ J_{q^*}^{E^*\to E^{**}}(J_q^{E\to E^*}(u))=u \]
holds cf. \cite[Lemma 2.63]{SKHK12}.
Note that in smooth and uniformly convex spaces (such as $L^p$ with $p\in(1,\infty)$), convergence with respect to the Bregman distance implies convergence with respect to the norm and vice versa, cf., e.g. \cite[Theorem 2.60]{SKHK12}
\be{equivalence}
D_q(\usol,u^k)\stackrel{k\to\infty}{\to}0 \ \Leftrightarrow \ D_q(u^k,\usol)\stackrel{k\to\infty}{\to}0 \ \Leftrightarrow \ \normE{u^k-\usol}\stackrel{k\to\infty}{\to}0.
\ee


While tools like the Bregman distance have only relatively recently been applied in the context of regularization, some of the fundamental concepts we use are still those from the seminal papers \cite{Natt77,VH85}.
In  \cite{Natt77}, which partly also works with general Banach spaces, error estimates for the general projection method \eqref{projmeth}
rely on the norms of the linear operator $B_n:F\to E_n$ mapping $\fdel$ to a solution of \eqref{projmeth}, as well as the special projection $\tilde{P}_n:E\to E_n$, $\tilde{P}_nu=B_nAu$. Note that well-definedness of these operators can be shown under certain conditions, see, e.g., \eqref{dim}, \eqref{nsp} below.
As a matter of fact, it is readily checked that for $u_n$ defined by \eqref{projmeth}, the error estimate
\[
\|u_*-u_n\|\leq (1+\|\tilde{P}_n\|)\,\mbox{dist}(u_*,E_n)+\|B_n\|\delta
\]
holds, which splits the total error into an approximation error term and a term bounding the noise propagation.
Error estimates of this type will enable the construction of convergent parameter choice rules also here and the concepts of quasi-optimality (uniform boundedness of $\|\tilde{P}_n\|$) and robustness (uniform boundedness of $\|B_n\| \alpha_n^{-1}$ with $\alpha_n=\sup_{u_n\in E_n,\,\|Au_n\|=1}\|u_n\|$) can be recovered in the boundedness conditions \eqref{bd}, \eqref{bd1}, \eqref{bdDP}, \eqref{kappa*appr0}, \eqref{kappa2*appr0}.

Note that computing general projection, least squares and least error  approximations in general Banach spaces might not be trivial. The reader is referred  e.g., to \cite[Section 3]{SchRieSch12}, \cite{SchSchLou08} for some iterative methods  (Landweber type, sequential subspace optimization) in uniformly convex and smooth Banach spaces.

\medskip


This work is organized as follows. Well-definedness, stability and convergence with a priori and a posteriori choices of  the dimension parameter are shown for the general projection method, the least squares method and the least error method in  Section \ref{sec_genproj}, \ref{sec_leastsquares} and  \ref{sec_leasterror}, respectively. This theory has, of course, its limitations and can approach problems in various couples of smaller or larger  function spaces $E$ and $F$,  as shortly outlined in   Section \ref{sec_space}. Some applications are  discussed in Section \ref{sec_applic}. Namely, analytical considerations and numerical tests are provided for a collocation method applied to a Volterra integral equation in one dimension space. 


\section{The general projection method}\label{sec_genproj}

Throughout this section, $E_n\subseteq E$ and $Z_n\subseteq F^*$, $n\in\N$, are finite dimensional subspaces.

\subsection{Well-definedness}
The following Lemma gives conditions for well-definedness of $u_n$ according to \eqref{projmeth}.
\begin{lemma}\label{welldef_projmeth}
Let 
\be{dim}
\mbox{dim}(E_n)=\mbox{dim}(Z_n)
\ee 
and 
\be{nsp}
\mathcal{N}(Q_nA)\cap E_n=\{0\}
\ee
hold. Then \eqref{projmeth} is uniquely solvable for any $\fdel\in F$.
\end{lemma}
\begin{proof}
Since \eqref{projmeth} is a finite dimensional linear system, with \eqref{dim}, unique solvability for any right hand side is equivalent to uniqueness, i.e., to the condition 
\[
\Bigl(w_n\in E_n \mbox{ and } \forall z_n\in Z_n \, : \ \dupF{z_n,A w_n}=0  \Bigr) \ \Rightarrow \ w_n=0
\]
This is the same as \eqref{nsp}.
\end{proof}

\subsection{Stability}

For stating stability we will make use of the following quantity:
\be{kappatilde}
\begin{aligned}
\tilde{\kappa}_n:=&
\sup_{w_n\in E_n,w_n\neq 0}\frac{\normE{w_n}}{\sup_{z_n\in Z_n ,\normFd{z_n}=1}\dupF{z_n,Aw_n}} \\
=&\frac{1}{\min_{w_n\in E_n,\normE{w_n}=1} \max_{z_n\in Z_n ,\normFd{z_n}=1} 
\dupF{z_n,Aw_n}} \\
=&\max_{w_n\in E_n, \normE{w_n}=1}\frac{1}{\norm{Q_nAw_n}_{Z_n^*}},
\end{aligned}
\ee
that is finite under conditions \eqref{dim}, \eqref{nsp}.
In case $\tau_n$ defined as
\be{tau_n}
\tau_n :=\sup_{w_n\in E_n,w_n\neq 0}\frac{\normF{Aw_n}}{\norm{Q_nAw_n}_{Z_n^*}}
\ee
is finite, which, e.g., is ensured by \eqref{nsp}, one can bound $\tilde{\kappa}_n$ by means of 
the simpler quantity 
\be{kappa_n}
\kappa_n:=\sup_{w_n\in E_n}\frac{\normE{w_n}}{\normF{Aw_n}}
=\max_{w_n\in E_n, \normE{w_n}=1}\frac{1}{\normF{Aw_n}}.
\ee

\begin{lemma}\label{kappa-inequality}
Suppose that 
\eqref{dim} and \eqref{nsp} hold.
Then the operator $A_n:=Q_nA|_{E_n}:E_n \to Z_n^*$ has an inverse and
\be{kappa_tau_n}
\kappa_n \leq \|A_n^{-1}\|=\tilde{\kappa}_n \leq \tau_n \kappa_n. 
\ee
\end{lemma}
\begin{proof}
For any $w_n\in E_n$ we have $\normF{A_n w_n}\geq \tau_n^{-1} \normF{Aw_n} \geq \tau_n^{-1} \kappa_n^{-1} \normE{w_n}$. Let $v_n\in E_n$ be an element for which in \eqref{kappa_n} the maximum is attained. Then by \eqref{Qn1} we have $\normF{A_n v_n} \leq \normF{A v_n}= \kappa_n^{-1} \normE{v_n}$.
\end{proof}

\begin{remark} \label{kappafinite}

Under  conditions \eqref{dim} and \eqref{nsp} of Lemma \ref{welldef_projmeth} one has  $\tilde{\kappa}_n<\infty$, since one takes the supremum over the unit sphere, which is compact in the finite dimensional spaces under consideration. 
The definition of the reciprocal of the stability factor $\tilde{\kappa}_n$ in the general projection method reveals the relation to Ladyshenskaja-Babuska-Brezzi (or inf-sup) conditions used for showing well-posedness of Petrov-Galerkin discretizations of partial differential equations.

\end{remark}

\medskip

For the general projection method we get:
\begin{lemma}\label{stability_projmeth}
Let the assumptions of Lemma \ref{welldef_projmeth} be satisfied and consider, for $f_1, f_2\in F$ the solutions of 
\[
u_{n,i}\in E_n \mbox{ and } \forall z_n\in Z_n \, : \ \dupF{z_n,A u_{n,i}}=\dupF{z_n,f_i} \quad i=1,2\,. 
\]
Then the estimate 
\[
\normE{u_{n,1}-u_{n,2}}\leq \tilde{\kappa}_n \normF{f_1-f_2}
\]
holds.
\end{lemma}
\begin{proof}
According to Lemma \ref{welldef_projmeth}, the solutions $u_{n,i}, i=1,2\,$  are well defined. Then,
the difference $\hat{u}_n=u_{n,1}-u_{n,2}$ satisfies 
\[
\hat{u}_n\in E_n \mbox{ and } \forall z_n\in Z_n \, : \ \dupF{z_n,A \hat{u}_n}=\dupF{z_n,f_1-f_2}. 
\]
Therefore, by definition of $\tilde{\kappa}_n$ and $u_{n,1}-u_{n,2}\in E_n$ (due to linearity of the space $E_n$) we have
\[
\begin{aligned}
\normE{u_{n,1}-u_{n,2}}\leq& \tilde{\kappa}_n \sup_{z_n\in Z_n ,\normFd{z_n}=1}\dupF{z_n,A(u_{n,1}-u_{n,2})}\\
=&\tilde{\kappa}_n \sup_{z_n\in Z_n ,\normFd{z_n}=1}\dupF{z_n,f_1-f_2}\leq \tilde{\kappa}_n \normF{f_1-f_2}
\end{aligned}
\]
\end{proof}

\subsection{Convergence with a priori choice of $n$}

\begin{theorem}\label{conv_projmeth_apriori}
Let for all $n\in\N$, the assumptions of Lemma \ref{welldef_projmeth} be satisfied and let $u_n$ be defined by the projection method \eqref{projmeth}. Additionally, we assume that 
there exists a sequence of approximations $(\hat{u}_n)_{n\in\N}$, $\hat{u}_n\in E_n$, satisfying the convergence conditions
\be{convPn}
\normE{\usol-\hat{u}_n}\to0\mbox{ as }n\to\infty
\ee
and 
\be{bd_conv}
\tilde{\kappa}_n 
\sup_{z_n\in Z_n ,\normFd{z_n}=1}\dupF{z_n,A(\usol-\hat{u}_n)}
\to0\mbox{ as }n\to\infty\,.
\ee

Then for exact data $\delta=0$ we have convergence 
\[\normE{u_n-\usol}\to0 \mbox{ as }n\to\infty\,.\]
For noisy data and with the dimension $n=n(\delta)$ chosen such that
\be{aprioristoppingrule}
n(\delta)\to\infty\mbox{ and }\tilde{\kappa}_{n(\delta)} \delta\to0 \mbox{ as }\delta\to0
\ee
we have convergence  
\[\normE{u_{n(\delta)}-\usol}\to0 \mbox{ as }\delta\to0\,.\]
\end{theorem}
\begin{proof}
For any $w_n\in E_n$ we have, by definition \eqref{kappatilde} of $\tilde{\kappa}_n$ and $u_n-w_n\in E_n$ (here linearity of the space $E_n$ is used), that 
\be{est_conv_projmeth}
\begin{aligned}
&\normE{u_n-\usol}\leq\normE{\usol-w_n}+\normE{u_n-w_n}\\
&\leq\normE{\usol-w_n}+ \tilde{\kappa}_n \sup_{z_n\in Z_n ,\normFd{z_n}=1}\dupF{z_n,A(u_n-w_n)}\\
&\leq\normE{\usol-w_n}+ \tilde{\kappa}_n \sup_{z_n\in Z_n ,\normFd{z_n}=1}\Bigl(\dupF{z_n,A(u_n-\usol)}+\dupF{z_n,A(\usol-w_n)}\Bigr)\\
&=\normE{\usol-w_n}+ \tilde{\kappa}_n \sup_{z_n\in Z_n ,\normFd{z_n}=1}\Bigl(\dupF{z_n,\fdel-f}
+\dupF{z_n,A(\usol-w_n)}\Bigr)\\
&\leq\normE{\usol-w_n}+ \tilde{\kappa}_n \Bigl( \delta + \sup_{z_n\in Z_n ,\normFd{z_n}=1}\dupF{z_n,A(\usol-w_n)} \Bigr)
\end{aligned}
\ee
where we have also used linearity of $A$.
Inserting $w_n=\hat{u}_n$ and using \eqref{convPn}, \eqref{bd_conv}, 
together with our assumptions on the choice of $n(\delta)$ we therefore immediately get the assertions in both cases $\delta=0$ and $\delta>0$.
\end{proof}
\begin{remark}\label{approx}
i) The approximation property \eqref{convPn} holds, e.g.,  when the subspaces $E_n$ are chosen according to \eqref{Espace}.

ii) Under conditions \eqref{P_n} and 
{
\be{convPPn}
\normE{\usol-P_n\usol}\to0\mbox{ as }n\to\infty,
\ee
}
on some sequence of operators $\{P_n:E\to E_n\,, \ n \in \N\}$, 
the uniform boundedness condition 
\be{bd}
\exists C<\infty \, \forall n\in\N \, : \, 
\tilde{\kappa}_n \sup_{w\in E,\normE{w}=1} 
\sup_{z_n\in Z_n ,\normFd{z_n}=1}\dupF{z_n,A(I-P_n)w}
\leq C 
\ee
is sufficient for \eqref{bd_conv} and by \eqref{est_conv_projmeth} with $w_n=P_n\usol$ yields the estimate
\be{est_projmeth}
\normE{u_n-\usol}\leq (1+C) \normE{\usol-P_n\usol} +\tilde{\kappa}_n \delta\,.
\ee 

In the context of Petrov Galerkin discretizations of PDEs, estimate \eqref{est_conv_projmeth} is known as Strang's First Lemma.

\end{remark}

\subsection{Convergence with a posteriori choice of $n$ -- the discrepancy principle}

\begin{theorem}\label{conv_projmeth_dp}
Let  the assumptions of Lemma \ref{welldef_projmeth} be satisfied for all $n\in\N$  and let $u_n$ be defined by the projection method \eqref{projmeth}. We also assume that there exists a sequence of approximations $(\hat{u}_n)_{n\in\N}$, $\hat{u}_n\in E_n$, satisfying \eqref{convPPn} and the conditions
\be{bd1_conv}
\kappa_n \sup_{z_n\in Z_n ,\normFd{z_n}=1}\dupF{z_n,A(\usol-\hat{u}_n)}\to0\mbox{ as }n\to\infty
\ee
\be{bdDP_conv}
\kappa_{n+1} \normF{(I-Q_n')A(\usol-\hat{u}_n)} \to0\mbox{ as }n\to\infty\,.
\ee
Additionally, we assume that there exists $\tau<\infty$ such that
\be{tau}
\tau \sup_{z_n\in Z_n ,\normFd{z_n}=1}\dupF{z_n,Aw_n}\geq  \normF{Aw_n}, \ \forall w_n\in E_n,
\ee
i.e., $\tau_n\leq \tau$ for all $n\in\N$, where $\tau_n$ is defined by \eqref{tau_n}.

Denote $d_{DP}(n)=\normF{Au_n-\fdel}, n\in\N$.
Let $b>1+\tau$ be fixed  and for $\delta>0$,  let $n=n_{DP}(\delta)$ be the first index such that
\be{nDP}
d_{DP}(n)\leq b\delta. 
\ee

Then $n_{DP}(\delta)$ is finite.\\
Moreoever, $u_{n_{DP}(\delta)}\to\usol$ as $\delta\to0$ subsequentially in the following sense: There exists a convergent subsequence and the limit of every convergent subsequence solves \eqref{Auf}; if $\usol$ is unique, then $\normE{u_{n_{DP}(\delta)}-\usol}\to0$ as $\delta\to0$.
\end{theorem}

\begin{proof}
By Lemma \ref{kappa-inequality} and assumption \eqref{tau} we can use $\kappa_n$ as in \eqref{kappa_n} instead of $\tilde{\kappa}_n$ as in \eqref{kappatilde} here.\\
For any $n$ let $w_n\in E_n$ be such that $Q_n' f^{\delta}=Aw_n$.
From \eqref{tau} it follows that
\be{est_discr}
\begin{aligned}
\normF{Au_n-f^{\delta}} &\leq \normF{A(u_n-w_n)} + \normF{Aw_n-f^{\delta}} \\
&\leq \tau \sup_{z_n\in Z_n ,\normFd{z_n}=1}\dupF{z_n,A(u_n-w_n)} + \normF{Aw_n-f^{\delta}}\\
&= \tau \sup_{z_n\in Z_n ,\normFd{z_n}=1}\dupF{z_n,f^{\delta}-Aw_n} + \normF{Aw_n-f^{\delta}} \\
&\leq (\tau+1 )\, \mbox{dist}(\fdel,AE_n).
\end{aligned}
\ee
In particular, since 
\[
\begin{aligned}
\limsup_{n\to\infty} \mbox{dist}(\fdel,AE_n)&\leq \delta + \limsup_{n\to\infty} \mbox{dist}(A\usol,AE_n)\\
&\leq \delta+ \limsup_{n\to\infty}\norm{A(\usol-\hat{u}_n)} = \delta,
\end{aligned}
\]
by \eqref{convPn}, this implies that $n_{DP}(\delta)$ is finite.
 
If for some $\delta_k\to 0\, (k\to\infty)$ the discrepancy principle gives $n_{DP}(\delta_k)\leq N$, with $N\geq 0$, then the sequence $u_{n_{DP}(\delta_k)}$ lies in a finite-dimensional subspace -- the linear hull of $E_n$, $n=0,\ldots,N$. Boundedness and therefore relative compactness of $(u_{n_{DP}(\delta_k)})$ follows from \eqref{est_conv_projmeth} (e.g., with $w_n=0$). Since $\normF{A u_{n_{DP}(\delta_k)}-f^{\delta_k}}\leq b\delta_k$, then $ A u_{n_{DP}(\delta_k)}\to f$ as $k\to \infty$. Hence $ (u_{n_{DP}(\delta_k)})$ has a convergent subsequence and the limit of every convergent subsequence solves \eqref{Auf}.

Otherwise,  $n_{DP}(\delta)$ will be larger than zero. In this case, let $m= n_{DP}(\delta)-1\geq 0$. 
For $n=m$ the inequality \eqref{nDP} does not hold and \eqref{est_discr} gives
\[
b\delta< \normF{Au_m-\fdel}\leq (\tau+1 )\, \mbox{dist}(\fdel,AE_m) \leq (\tau+1 ) (\delta+\mbox{dist}(f,AE_m)).
\]
Since $b>1+\tau$ we have
\begin{equation}\label{estQmp}
\frac{(b-1-\tau)\delta}{\tau+1}< \mbox{dist}(f,AE_m) = \normF{A\usol -Q_m'A\usol} = \normF{(I-Q_m')A(\usol-\hat u_m)}.
\end{equation}

Inserting this into \eqref{est_conv_projmeth} with $w_n=\hat{u}_n$, $n=n_{DP}(\delta)$ and using \eqref{kappa_tau_n}, \eqref{convPn}, \eqref{bdDP_conv}, \eqref{bd1_conv}, we get convergence if $n_{DP}(\delta)\to\infty$ as $\delta\to0$.

\end{proof}

\begin{remark}
If  some sequence of  operators $\{P_n:E\to E_n\,, \ n \in \N\}$ satisfies \eqref{P_n} and \eqref{convPPn},
then \eqref{bd1_conv} and \eqref{bdDP_conv} follow from \eqref{convPn} for $\hat{u}_n=P_n\usol$ and from the uniform boundedness conditions
\be{bd1}
\exists C<\infty \, \forall n\in\N \, : \, 
\kappa_n \sup_{w\in E,\normE{w}=1} 
\sup_{z_n\in Z_n ,\normFd{z_n}=1}\dupF{z_n,A(I-P_n)w}
\leq C,\,
\ee
\be{bdDP}
\begin{aligned}
\exists C_1<\infty \, \forall n\in\N \, : \, 
&\kappa_{n+1} \sup_{w\in E,\normE{w}=1}\normF{(I-Q_n')A (I-P_n)w} \leq C_1.
\end{aligned}
\ee
If additionally $I-Q_n'$ is homogeneous, one has    by  \eqref{estQmp} and by homogeneity of   $I-P_n$  
\be{kappa-tau-est}
\begin{aligned}
\frac{\kappa_{m+1} (b-1-\tau)\delta }{\tau+1}&\leq \kappa_{m+1} \sup_{w\in E,\normE{w}=1}\normF{(I-Q_m')A(I-P_m) w} \normE{(I-P_m) \usol} \\
&\leq C_1 \normE{(I-P_m) \usol}.
\end{aligned}
\ee
Hence, by \eqref{est_projmeth} and Lemma \ref{kappa-inequality}, we obtain  the error estimate
\be{eeDP}
\begin{aligned}
&\normE{u_{n_{DP}(\delta)}-\usol}\\
&\leq (1+\tau C)\normE{(I-P_{n_{DP}(\delta)})\usol}+ \frac{C_1 \tau(\tau+1)}{b-\tau-1}\normE{(I-P_{n_{DP}(\delta)-1})\usol} \,,
\end{aligned}
\ee
in case ${n_{DP}(\delta)}\geq 1$.

Note that conditions \eqref{bd1} and \eqref{bdDP} correspond to  $m=1$ in condition (1.27) of \cite{VH85}. 
\end{remark}


\section{The least squares method}\label{sec_leastsquares}

Throughout this section,   $E_n\subseteq E$ is a  finite dimensional subspace. We show below that the least squares method is well-defined and converges to a solution under a priori and a posteriori choices for the discretization dimension. 

\subsection{Well-definedness}
\begin{lemma}\label{welldef_leastsquares}
Let 
\be{nsp1}
\mathcal{N}(A)\cap E_n=\{0\}.
\ee
Then the set of minimizers  $\mbox{argmin}\{ \normF{A \tilde{u}_n-\fdel}\, : \, \tilde{u}_n\in E_n\}$ is nonempty. 
If, for some $q\geq1$, the functional 
\be{PhiF}
\Phi_F: g\mapsto\tfrac{1}{q}\normF{g}^q
\ee
is strictly convex, then the minimizer is unique.
\end{lemma}
\begin{proof}
The finite dimensional linear subspace $E_n$ is reflexive, closed, convex and nonempty. The cost functional $j:\tilde{u}_n\mapsto\normF{A \tilde{u}_n-\fdel}$ is convex, weakly lower semicontinuous, bounded from below. It is also coercive, since the minimum $\kappa_n^{-1}=\min\{\normF{A \hat{u}_n}\, : \ \hat{u}_n\in E_n,\,\norm{\hat{u}_n}=1\}$ exists  on the finite dimensional hence compact unit sphere and is positive by condition \eqref{nsp1}, hence boundedness of some sequence $(\normF{A \tilde{u}_n^k-\fdel})_{k\in\N}$ implies boundedness of $(\normE{\tilde{u}_n^k})_{k\in\N}$ as follows:
\[
\normE{\tilde{u}_n^k}
\leq \sup_{\tilde{u}_n\in E_n,\tilde{u}_n\neq 0}\frac{\normE{\tilde{u}_n}}{\normF{A\tilde{u}_n}} \normF{A\tilde{u}_n^k}
=\kappa_n \normF{A\tilde{u}_n^k} \leq \kappa_n (\normF{A\tilde{u}_n^k-\fdel}+\normF{\fdel})\,.
\]
Thus we can conclude existence of a minimizer.\\
Minimizing $j$ over $E_n$ is obviously equivalent to minimizing $\tfrac{1}{q}j^q$ over $E_n$.
Moreover, strict convexity of the functional $\Phi_F$ by \eqref{nsp1} transfers to the functional $\tfrac{1}{q}j^q:u_n\mapsto \tfrac{1}{q}\normF{A \tilde{u}_n-\fdel}^q$ on $E_n$. This implies uniqueness.
\end{proof}

In the Hilbert space setting, the least squares  method can be shown to be a special case of the general projection method \eqref{projmeth} upon appropriate choice of the spaces  $E_n$.


\begin{lemma}\label{characterization_leastsquares}
Let \eqref{nsp1} hold and let  $u_n$ be defined  by the least squares method  \eqref{leastsquares}. Assume that, for some $q\geq1$, the functional 
\eqref{PhiF} is strictly convex,
 the duality mappings  satisfy 
\[ \forall g\in F\, : \ J_{q^*}^{F^*\to F^{**}}(J_q^{F\to F^*}(g))=g \]
and  $J_q^{F\to F^*}$ is Gateaux differentiable at $(Au_n-\fdel)$ with Gateaux derivative \\
$G_n:F\to F^*$, $G_n= (J_q^{F\to F^*})'(Au_n-\fdel)$.

Then  we have equivalence of \eqref{leastsquares} and \eqref{projmeth} by considering the linear space $Z_n =G_nAE_n$.
\end{lemma}
\begin{proof}
Using the identity $\normF{f}^{q-1}=\normFd{J_q^{F\to F^*}(f)}$, and the functional 
\[
\Phi_{F^*}:z\mapsto\tfrac{1}{q^*}\normFd{z}^{q^*}
\]
we have that \eqref{leastsquares} is equivalent to 
\[
u_n\in\mbox{argmin}\{ \Phi_{F^*}(J_q^{F\to F^*}(A \tilde{u}_n-\fdel))\, : \, \tilde{u}_n\in E_n\}
\]
The necessary and, by convexity, also sufficient condition for this optimality problem reads as
\[
\begin{aligned}
\forall w_n\in E_n \, : \ 
0&=\tfrac{d}{d u_n} \left(\Phi_{F^*}(J_q^{F\to F^*}(A u_n-\fdel))\right)[w_n]\\
&=\dupF{ \tfrac{d}{d u_n} \left(J_q^{F\to F^*}(Au_n-\fdel)\right) [w_n], 
\partial \Phi_{F^*}(J_q^{F\to F^*}(A u_n-\fdel)}\\
&=\dupF{(J_q^{F\to F^*})'(Au_n-\fdel)A w_n , J_{q^*}^{F^*\to F^{**}}(J_q^{F\to F^*}(A u_n-\fdel)}\\
&=\dupF{(J_q^{F\to F^*})'(Au_n-\fdel)A w_n , A u_n-\fdel}\,,
\end{aligned}
\]
which is \eqref{projmeth} with $Z_n =(J_q^{F\to F^*})'(Au_n-\fdel)AE_n$.
\end{proof}
\begin{remark}
Since $u_n$ from the definition of the operator $G_n$ is unknown, Lemma \ref{characterization_leastsquares} is only of theoretical use. Later on, it will enable us to conclude convergence from the respective result for general projection methods - see Corollary \ref{conv_leastsquares_apriori} below. For practical computation of $u_n$, the finite dimensional minimization problem \eqref{leastsquares} should be solved. 
\end{remark}

Note that the equality $\forall g\in F\, : \ J_{q^*}^{F^*\to F^{**}}(J_q^{F\to F^*}(g))=g$ required by the previous lemma holds, e.g., in reflexive spaces, cf. \cite{Cioranescu90}.

\subsection{Stability}

For the least squares method, the crucial quantity  in the stability estimate is $\kappa_n$ defined as in \eqref{kappa_n}.
As in Remark \ref{kappafinite}, under the conditions of  Lemma \ref{welldef_leastsquares}, we have $\kappa_n<\infty$.
Therewith we obtain the following stability result.

\begin{lemma}\label{stability_leastsquares}
Let all the assumptions of Lemma \ref{characterization_leastsquares}  be satisfied and consider, for $f_1, f_2\in F$ the solutions of 
\[
u_{n,i}\in\mbox{argmin}\{ \normF{A \tilde{u}_n-f_i}\, : \, \tilde{u}_n\in E_n\}
\quad i=1,2\,. 
\]
Then the estimate 
\[
\normE{u_{n,1}-u_{n,2}}\leq \kappa_n \normF{Q_n'f_1-Q_n'f_2}
\]
holds, where 
$Q_n'$ is some single valued selection of the metric projection onto the subspace $AE_n$. If $Q_n'$ is continuous, then $u_{n,1}$ depends continuously on $f_1$. 
\end{lemma}
\begin{proof}
The proof follows by the definition of $\kappa_n$ and the fact that $Au_{n,i}=Q_n'f_i$.
\end{proof}
\begin{remark}\label{Qpcontinuous}
The metric projection operator $Q_n'$ onto closed convex sets is single valued and continuous in uniformly convex Banach spaces (see, e.g., \cite{AN95}). Thus, the above result is applicable to the setting $F=L^p$ with $p\in (1,+\infty)$, but not to the space  $C(\overline{\Omega})$ in general. However, since the subspaces $AE_n$ are finite dimensional according to the rank-nullity theorem for linear mappings, one might work with continuous selections of the metric operators in this nonreflexive Banach space setting if those subspaces have certain properties - see, e.g.,  Theorem 6.34 in \cite{DS98}. More precisely, the metric projection $P_S$ onto an $n$-dimensional subspace $S$  of $C[a,b]$ admits a unique continuous selection if and only if every function $f\in S$, $f\neq 0$ has at most $n$ zeros and if every $f\in S$ has at most $n-1$ changes of sign. 
\end{remark}

\subsection{Convergence with a priori choice of $n$}

Together with Lemma \ref{characterization_leastsquares}, Theorem \ref{conv_projmeth_apriori} immediately implies convergence of the least squares method.
\begin{corollary}\label{conv_leastsquares_apriori}
Let   all the assumptions of Lemma \ref{characterization_leastsquares} be satisfied. 
Additionally, we assume that there exists a sequence of approximations $(\hat{u}_n)_{n\in\N}$, $\hat{u}_n\in E_n$, satisfying \eqref{convPn}, \eqref{bd_conv}, where $\tilde{\kappa}_n$ is defined as in \eqref{kappatilde} with $Z_n =(J_q^{F\to F^*})'(Au_n-\fdel)AE_n$.

Then for exact data $\delta=0$ we have convergence as $n\to\infty$
\[\normE{u_n-\usol}\to0 \mbox{ as }n\to\infty\,.\]
For noisy data and with the dimension $n=n(\delta)$ chosen according to \eqref{aprioristoppingrule}, we have convergence as $\delta\to0$ 
\[\normE{u_{n(\delta)}-\usol}\to0 \mbox{ as }\delta\to0\,.\]
\end{corollary}
Alternatively, we can also prove convergence directly:
\begin{theorem}\label{conv_leastsquares_apriori_gen}
Let condition \eqref{nsp1} be satisfied for all $n\in\N$.
Then an approximation $u_n$ according to the least squares method \eqref{leastsquares} exists 
and the error estimate
\be{ls_error}
\normE{u_n-\usol}\leq \inf_{w_n \in E_n} \{\normE{\usol-w_n
} +2 \kappa_n \normF{A\usol-Aw_n}\}+2\kappa_n \delta
\ee
holds.
If there exists a sequence of approximations $(\hat{u}_n)_{n\in\N}$, $\hat{u}_n\in E_n$, satisfying 
\eqref{convPn} and
\be{kappa*appr}
\kappa_n\normF{A(\usol-\hat{u}_n)}\to0 
 \mbox{ as } n\to\infty,
\ee 
then we have in case of exact data ($\delta=0$) convergence
\[\normE{u_n-\usol}\to0 \mbox{ as }n\to\infty\,\]
and in case of noisy data with the choice of $n=n(\delta)$ according to 
\be{n_choice}
n(\delta)\to\infty\mbox{ and }\kappa_{n(\delta)} \delta \to0 \mbox{ as }\delta\to0 
\ee
convergence as $\delta\to0$: 
\[\normE{u_{n(\delta)}-\usol}\to0 \mbox{ as }\delta\to0\,.\]
\end{theorem}
\begin{proof}
Let $w_n\in E_n$ be arbitrary. We have $\normF{Au_n -\fdel} \leq \normF{Aw_n -\fdel}$
due to the least squares property \eqref{leastsquares},  
therefore
\[
\begin{aligned}
\normF{Au_n-Aw_n} &\leq \normF{Au_n -\fdel} + \normF{Aw_n -\fdel} \\
&\leq 2 \normF{Aw_n -\fdel} \leq 2 (\normF{Aw_n -f} + \delta)
\end{aligned}
\]
and
\[
\begin{aligned}
\normE{u_n-\usol}&\leq\normE{\usol-w_n}+\normE{u_n-w_n}\\
&\leq\normE{\usol-w_n}+ \kappa_n \normF{Au_n-Aw_n}\\
&\leq\normE{\usol-w_n}+ 2\kappa_n(\normF{Aw_n-f}+\delta).
\end{aligned}
\]
In this error estimate $w_n \in E_n$ is arbitrary, hence \eqref{ls_error} holds. If some sequence of approximations $(\hat{u}_n)_{n\in\N}$ satisfies conditions \eqref{convPn} and \eqref{kappa*appr}, then
insertion of $w_n=\hat{u}_n$ into \eqref{ls_error} gives the convergence assertions. 
\end{proof}

\begin{remark} \label{LSconvcond}
Note that convergence condition \eqref{kappa*appr} is satisfied, if some sequence of  operators $\{P_n:E\to E_n\,, \ n \in \N\}$ satisfies conditions \eqref{P_n}, \eqref{convPPn} and
\be{kappa*appr0}
\exists C<\infty \, \forall n\in\N \, : \, 
\kappa_n \sup_{w\in H,\normE{w}=1}\normF{A (I-P_n)w} \leq C
 \ee
(compare \eqref{bd}, \eqref{bd1}). Namely,  the equality $I-P_n=(I-P_n)^2$ allows to estimate $\kappa_n \normF{A(I-P_n)\usol} \leq C\normE{(I-P_n)\usol}.$

\end{remark}

\subsection{Convergence with a posteriori choice of $n$ -- the discrepancy principle}

\begin{theorem}\label{conv_leastsquares_dp}
Let for all $n\in\N$ condition \eqref{nsp1} be satisfied so that $u_n$ according to the least squares method \eqref{leastsquares} exists.
Additionally, we assume that there exists a sequence of approximations $(\hat{u}_n)_{n\in\N}$, $\hat{u}_n\in E_n$, satisfying \eqref{convPn} and the condition
\be{kappa2*appr0_conv}
\kappa_{n} (\normF{A (\usol-\hat{u}_n)} + \normF{A (\usol-\hat{u}_{n-1})})\ \to 0 \mbox{ as } n\to\infty
\ee
Let $b>1$ be fixed and for $\delta>0$, let $n=n_{DP}(\delta)$ be the first index such that
\eqref{nDP}
holds.

Then for $\delta>0$ we have that $n_{DP}(\delta)$ is finite. 
Moreover, $u_{n_{DP}(\delta)}\to\usol$ as $\delta\to0$ subsequentially.
\end{theorem}
\begin{proof}
The proof is the same as for the Theorem \ref{conv_projmeth_dp}, but \eqref{est_discr} with $\tau=0$ is now trivial, using optimality of $u_n$ for the minimization problem \eqref{leastsquares}, and hence $\tau$ can be omitted in  formula \eqref{estQmp} (with $m=n_{DP}(\delta)-1)$). Note also that  we do not need now relations \eqref{kappa_tau_n} and that estimate \eqref{ls_error} with $w_n=\hat{u}_n$ is used instead of \eqref{est_conv_projmeth}. 
\end{proof}

\begin{remark}
If some sequence of  operators $\{P_n:E\to E_n\,, \ n \in \N\}$ satisfies \eqref{P_n} and \eqref{convPPn}, then the uniform boundedness condition
\be{kappa2*appr0}
\exists C<\infty \, \forall n\in\N \, : \, (\kappa_{n} +\kappa_{n+1}) \sup_{w\in H,\normE{w}=1}\normF{A (I-P_n)w} \leq C.
\ee
is sufficient for \eqref{kappa2*appr0_conv}.
\end{remark}

\section{The least error method}\label{sec_leasterror}

Throughout this section, $Z_n\subseteq F^*$ is a finite dimensional subspace. We establish  well-definedness and convergence of the least error method to a solution  under a priori and a posteriori choices for the discretization dimension.

\subsection{Well-definedness}
\begin{lemma}\label{welldef_leasterror}
\begin{enumerate}
\item[(i)]
Let $E$ be a Banach space in which the unit ball is weakly compact and assume that 
\be{nspd}
\mathcal{N}(A^*)\cap Z_n =\{0\}\,.
\ee

Then the set of minimizers $\mbox{argmin}\{ \normE{\tilde{u}} \, : \, \forall z_n\in Z_n \, : \ \dupF{z_n,A \tilde{u}}=\dupF{z_n,\fdel}\}$ is nonempty. 

\item[(ii)]
If additionally for some $q\geq1$, the functional 
\be{PhiH}
\Phi_E:\tilde{u}\mapsto\tfrac{1}{q}\normE{\tilde{u}}^q
\ee 
is strictly convex, then the minimizer $u_n$ of \eqref{leasterror} is unique, and so is the minimum-norm-solution $\udag$ of \eqref{Auf}.
\end{enumerate}
\end{lemma}


\begin{proof}
Condition \eqref{nspd} implies that the admissible set $E_n^{ad}=
\{ \tilde{u}\in E \, : \, \forall z_n\in Z_n \, : \ \dupF{z_n,A \tilde{u}}=\dupF{z_n,\fdel}\}$ is nonempty.
To see this, we apply the Closed Range Theorem to the linear operator $Q_nA:E\to Z_n^*$, whose finite dimensional range is obviously closed. Hence we have the identity 
\[
\begin{aligned}
Q_nAE=&\mathcal{N}((Q_nA)^*)^\bot\\
=&\{g_n\in Z_n^* \, \vert \, \forall z_n\in Z_n\, : \  (Q_nA)^*z_n=0\ \Rightarrow \ \dupZn{g_n,z_n}=0\}\\
=&\{g_n\in Z_n^* \, \vert \, \forall z_n\in Z_n\, : \  A^*z_n=0\ \Rightarrow \ \dupZn{g_n,z_n}=0\}\\
=& Z_n^*
\end{aligned}
\]
under condition \eqref{nspd}, since by definition of $Q_n$ we have $Q_n^*z_n=z_n$ for all $z_n\in Z_n$:
\[
\forall g\in F: \dupF{Q_n^*z_n,g}=\dupZn{Q_ng,z_n}=\dupF{z_n,g}.
\]
Thus, the equation $Q_nA\tilde{u}=f_n$ with $f_n=Q_n\fdel\in Z_n^*$ defining $E_n^{ad}$ is always solvable under condition \eqref{nspd}.

Due to our assumption on $E$, level sets of the cost function $j:\tilde{u}\mapsto \normE{\tilde{u}}$ are weakly compact. Moreover, $j$ is weakly lower semicontinuous and bounded from below. This implies existence of a minimizer, which in case of strict convexity of the cost function $\tfrac{1}{q}j^q=\Phi_E$ is obviously unique.
\end{proof}

Also the least error method is to some extent a special case of \eqref{projmeth}. However, different from the Hilbert space situation, the ansatz space might be nonlinear in general Banach spaces.

\begin{lemma}\label{characterization_leasterror}
Let the conditions of Lemma \ref{welldef_leasterror} (i) be satisfied and let additionally, for some $q>1$, the norm functional $u\mapsto\tfrac{1}{q}\normE{u}^q$ be Frechet differentiable, and the single valued duality mapping $J_q^{E\to E^*}$  be invertible. 

Then \eqref{leasterror} and \eqref{projmeth} are equivalent, when $E_n=(J_q^{E\to E^*})^{-1}(A^*Z_n )$.
\end{lemma}
\begin{proof}
By convexity and Frechet differentiability of the cost function as well as linearity of the constraints, optimality in \eqref{leasterror} is equivalent to existence of a Lagrange multiplier $\lambda\in {\mathbb{R}}^{m_n}$ with $m_n=dim\, Z_n$ such that stationarity for the Lagrange function 
\[
\begin{aligned}
L:H\times  {\mathbb{R}}^{m_n} \rightarrow&\mathbb{R}\\
(\tilde{u},\lambda)\mapsto&\tfrac{1}{q}\normE{\tilde{u}}^q +\sum_{i=1}^{m_n} \lambda_i\dupF{z_{n,i},A \tilde{u}-\fdel}\\
&=\tfrac{1}{q}\normE{\tilde{u}}^q + \dupE{A^*(\sum_{i=1}^{m_n}\lambda_i z_{n,i}),\tilde{u}}-\dupF{\sum_{i=1}^{m_n}\lambda_i z_{n,i},\fdel}
\end{aligned}
\]
holds, where $Z_n=span\{z_{n,i}, i\in\{1,2,...,m_n\}\}$. That is, there exists $\bar \lambda\in {\mathbb{R}}^{m_n}$  such that 
\[
\begin{aligned}
J_q^{E\to E^*}(u_n)+A^*v_n =0 \mbox{ and }\\
\forall z_n\in Z_n \, : \ \dupF{z_n,A u_n-\fdel}=0,
\end{aligned}
\]
where $v_n=\sum_{i=1}^{m_n}\bar\lambda_iz_{n,i}$.
The first of these two equations with invertibility of $J_q^{E\to E^*}$ yields that \eqref{leasterror} is equivalent to \eqref{projmeth} with $E_n=(J_q^{E\to E^*})^{-1}(A^*Z_n)$.\\
The  implication \eqref{projmeth} $\Rightarrow$\eqref{leasterror}  can be shown also in a variational manner, by exploiting duality mapping properties. We include the alternative proof here, for the sake of completeness. Thus, assume that $u_n$ satisfies \eqref{projmeth} with $E_n=(J_q^{E\to E^*})^{-1}(A^*Z_n)$ and let $u$ be an arbitrary element of the feasible set $E_n^{ad}=
\{ \tilde{u}\in E \, : \, \forall z_n\in Z_n \, : \ \dupF{z_n,A \tilde{u}}=\dupF{z_n,\fdel}\}$.
Then we can write $u_n=(J_q^{E\to E^*})^{-1}(A^*v_n)$ for some $v_n$ and insert $z_n=v_n$ to obtain the identity $\dupF{v_n,Au_n}=\dupF{v_n,\fdel}$, which together with feasibility of $u$ yields 
\[
\begin{aligned}
&\dupE{J_q(u_n),u_n}
=\dupE{A^*v_n,u_n}
=\dupF{v_n,Au_n}\\
&=\dupF{v_n,\fdel}=\dupF{v_n,Au}
=\dupF{A^*v_n,u}\leq \normEd{A^*v_n}\normE{u}\,.
\end{aligned}
\]
On the other hand, we have
\[
\dupE{J_q(u_n),u_n}=\normEd{A^*v_n}\normE{u_n}\,,
\]
thus altogether
\[
\normE{u_n}\leq \normE{u}\,.
\]
\end{proof}

\begin{remark}
Note that $E_n=(J_q^{E\to E^*})^{-1}(A^*Z_n )$ is not necessarily a linear space, though. So in the proof of stability and convergence we cannot resort to the respective results on the general projection method, but have to carry out separate proofs for the least error method, see Lemma \ref{stability_leasterror} and Theorem \ref{conv_leasterror_apriori} below.
\end{remark}

In the sequel one can see that the ''least error'' method deserves its name in the Banach space setting, too.

\begin{theorem}\label{rem_leasterror} Let 
  $\usol$  be some solution of \eqref{Auf}. Then, for any $n\in\N$, the minimizer $u_n$ defined by \eqref{leasterror} in case $f^\delta=f$ attains the least error in $E_n=(J_q^{E\to E^*})^{-1}(A^*Z_n)$ measured  with respect to the Bregman distance, that is,
$$D(\usol,u_n)\leq D(\usol,u),\quad\forall u\in E_n=(J_q^{E\to E^*})^{-1}(A^*Z_n).$$

\begin{proof}
 In the case of exact data, equation \eqref{projmeth} can be written as 
\be{exact}
\dupF{z_n,A u_n-A\usol}=0, \forall z_n\in Z_n.
\ee 
Let $u\in(J_q^{E\to E^*})^{-1}A^*v$ with $v\in Z_n$ be an arbitrary element of $E_n=(J_q^{E\to E^*})^{-1}(A^*Z_n)$. Then one has 
\begin{align*}
&D(\usol,u_n)+D(u_n,u)-D(\usol,u) = \dupF{J_q^{E\to E^*}(u)-J_q^{E\to E^*}(u_n),\usol-u_n}\\
&= \dupF{A^*v-A^*v_n,\usol-u_n}
= \dupF{v-v_n,f-Au_n}
= 0,
\end{align*}
as $v,v_n\in Z_n$ satisfy \eqref{exact}.  Since $D(u_n,u)$ is nonnegative, this implies the desired inequality,  showing that $u_n$ is the Bregman projection of $\usol$ onto $E_n$. \end{proof}
\end{theorem}
 
\subsection{Stability}
A stability result for the least error method can be formulated by using
\be{kappa}
\begin{aligned}
\hat{\kappa}_n:=&\sup_{z_{n,1},z_{n,2}\in Z_n }\frac{\normFd{z_{n,1}-z_{n,2}}}{\Bigl(D_q^{sym}\bigl((J_q^{E\to E^*})^{-1}(A^*z_{n,1}),(J_q^{E\to E^*})^{-1}(A^*z_{n,2})\bigr)\Bigr)^\frac{1}{q^*}}\\
=& \max_{z_{n,1},z_{n,2}\in Z_n, \normFd{z_{n,1}}=1, \normFd{z_{n,2}}\leq1}
\frac{\normFd{z_{n,1}-z_{n,2}}}{\Bigl(D_{q^*}^{sym}(A^*z_{n,1},A^*z_{n,2})\Bigr)^\frac{1}{q^*}}\\
\kappa_n^*:=&\sup_{z_n\in Z_n }\frac{\normFd{z_n}}{\normEd{A^*z_n}}
=\frac{1}{\min_{z_n\in Z_n ,\normFd{z_n}=1} \normEd{A^*z_n}}
\end{aligned}
\ee
Again, as in Remark \ref{kappafinite}, one sees that $\hat{\kappa}_n$ and $\kappa_n^*$ are finite under the conditions of Lemma \ref{welldef_leasterror}, in particular, condition \eqref{nspd}.

\begin{lemma}\label{stability_leasterror}
Let the assumptions of Lemma  \ref{characterization_leasterror} be satisfied and consider, for $f_1, f_2\in F$ the solutions of 
\[
u_{n,i}\in\mbox{argmin}\{ \normE{\tilde{u}} \, : \, \forall z_n\in Z_n \, : \ \dupF{z_n,A \tilde{u}}=\dupF{z_n,f_i}\} \quad i=1,2\,. 
\]
Then the estimate 
\[
D_q^{sym}(u_{n,1},u_{n,2})^{\frac{1}{q}}\leq\hat{\kappa}_n \normF{f_1-f_2}
\] 
holds; in particular, if $E$ is a $q$-convex space, then one has
\[
\normE{u_{n,1}-u_{n,2}}\leq \hat{\kappa}_n \normF{f_1-f_2}\,.
\]
If additionally $E$ is $s$-smooth, then one has 
\[
\normE{u_{n,1}-u_{n,2}}^{q-s+1}\leq
\frac{C_s}{c_q} \max\{\normE{u_{n,1}},\normE{u_{n,2}}\}^{q-s} \kappa_n^* \normF{f_1-f_2}\,.
\]
for some constants $c_q,C_s$ independent of $n$.
\end{lemma}
\begin{proof} According to Lemma  \ref{welldef_leasterror}, the solutions $u_{n,i},\,i=1,2$, are well defined.
Lemma \ref{characterization_leasterror} implies existence of $v_{n,i}\in Z_n $ such that $J_q^{E\to E^*}(u_{n,i})=A^*v_{n,i}$, $i=1,2$. Therefore, 
we get the identity
\[
\begin{aligned}
D_q^{sym}(u_{n,1},u_{n,2})=&
\dupE{J_q^{E\to E^*}(u_{n,2})-J_q^{E\to E^*}(u_{n,1}),u_{n,2}-u_{n,1}}\\
=&\dupE{A^*(v_{n,2}-v_{n,1}),u_{n,2}-u_{n,1}}
=\dupF{v_{n,2}-v_{n,1},f_2-f_1}\\
\leq&\hat{\kappa}_n \Bigl(D_q^{sym}(u_{n,1},u_{n,2})\Bigr)^{\frac{1}{q^*}} \normF{f_2-f_1}
\end{aligned}
\]
Similarly, in the $q$-convex and $s$-smooth case, which implies
\be{qconvex}
D_q^{sym}(\tilde{u},u)\geq D_q(\tilde{u},u)\geq c_q\normE{\tilde{u}-u}^q 
\ee
\[
\normEd{J_q^{E\to E^*}(\tilde{u})-J_q^{E\to E^*}(u)}\leq C_s \max\{\normE{\tilde{u}},\normE{u}\}^{q-s}
\normE{\tilde{u}-u}^{s-1}
\]
for some constants $c_q,C_s>0$ and all $\tilde{u},u\in H$ (see, e.g., \cite[Lemma 2.7]{BKMSS07}, \cite[Theorem 2.42]{SKHK12}), we get
\[
\begin{aligned}
&c_q\normE{u_{n,1}-u_{n,2}}^q\leq D_q^{sym}(u_{n,1},u_{n,2})
=\dupF{v_{n,2}-v_{n,1},f_2-f_1}\\
&\leq\normFd{v_{n,2}-v_{n,1}}\normF{f_2-f_1}\leq\kappa_n^*\normFd{A^*v_{n,2}-A^*v_{n,1}}\normF{f_2-f_1}\\
&\leq\kappa_n^* C_s \max\{\normE{u_{n,1}},\normE{u_{n,2}}\}^{q-s} \normE{u_{n,1}-u_{n,2}}^{s-1} \normF{f_2-f_1}\,.
\end{aligned}
\]
\end{proof}
Note that for $p\in(1,\infty)$, $L^p(\Omega)$ is $\max\{p,2\}$-convex and $\min\{p,2\}$-smooth, see, e.g, \cite[Example 2.47]{SKHK12}.

\subsection{Convergence with a priori choice of $n$}
For the least error method, due to possible nonlinearity of the space $E_n$ according to Lemma \ref{characterization_leasterror}, convergence cannot be directly concluded from Theorem \ref{conv_projmeth_apriori}. We obtain the following result with a priori discretization level choice.
\begin{theorem}\label{conv_leasterror_apriori}
Let $E$ be a Banach space in which the unit ball is weakly compact and assume that, for some $q\geq1$, the functional 
\eqref{PhiH} is strictly convex and  Frechet differentiable, and the single valued duality mapping $J_q^{E\to E^*}$  is invertible. 

Let $u_n$ be defined by the least error method \eqref{leasterror}, where the operator $A$ is assumed to satisfy
\eqref{nspd} and 
\be{convQn}
\forall z\in F^*\, : \ \inf_{z_n\in Z_n} \normEd{A^*(z-z_n)}\to0\mbox{ as }n\to\infty\,.
\ee
Then the minimum-norm-solution $\udag$ of \eqref{Auf} is unique and 
for exact data ($\delta=0$) we have convergence 
\[D_q(\udag_n,\udag)\to0 \mbox{ as }n\to\infty\,.\]
If, additionally,  the space $E$ is smooth and uniformly convex, one has 
\[\normE{\udag_n-\udag}\to0 \mbox{ as }n\to\infty\,\]
for exact data, while
for noisy data and with the dimension $n=n(\delta)$ chosen such that
\be{aprioristoppingrule1}
n(\delta)\to\infty\mbox{ and }\hat\kappa_{n(\delta)}\delta\to0 \mbox{ as }\delta\to0\,,
\ee
we have convergence 
\[\normE{u_{n(\delta)}-\udag}\to0 \mbox{ as }\delta\to0\,.\]
.

\end{theorem}
\begin{proof}
Let $\udag_n$ be the well defined elements   (due to Lemma \ref{welldef_leasterror}) 
\be{udagn}
\udag_n\in\mbox{argmin}\{ \normE{\tilde{u}} \, : \, \forall z_n\in Z_n \, : \ \dupF{z_n,A \tilde{u}}=\dupF{z_n,f}\}\,,
\ee
i.e., $u_n$ with exact data. Then the following holds for any solution $\usol$ to \eqref{Auf},
\be{udagnuast}
\normE{\udag_n}\leq \normE{\usol}.
\ee
By the assumed weak compactness of the unit ball in $E$, the sequence $(\udag_n)_{n\in\N}$ has a weakly convergent subsequence $(\udag_{n_l})_{l\in\N}$ whose limit $u$ solves \eqref{Auf}, since $A$ is weakly continuous and by $\dupF{z_{n_l},A\udag_{n_l}-f}=0$ for all $z_{n_l}\in Z_{n_l}$ and \eqref{convQn} we have
\[
\begin{aligned}
\forall z\in F^*\, : \
\dupF{z,A\udag_{n_l}-f}=&\inf_{z_{n_l}\in Z_{n_l}}\dupF{z-z_{n_l},A\udag_{n_l}-f)}\\
=&\inf_{z_{n_l}\in Z_{n_l}}\dupF{z-z_{n_l},A(\udag_{n_l}-\usol)}\\
\leq&\inf_{z_{n_l}\in Z_{n_l}}\normEd{A^*(z-z_{n_l})} 2 \normE{\usol} \to 0 \mbox{ as }l\to\infty\,.
\end{aligned}
\]
Moreover, by \eqref{udagnuast} and weak lower semicontinuity of the norm, this limit $u$ satisfies $\normE{u}\leq \normE{\usol}$ for any solution $\usol$ of \eqref{Auf}, thus it has to coincide with the unique minimum-norm-solution $\udag$. A subsequence-subsequence argument yields weak convergence of the whole sequence $(\udag_n)_{n\in\N}$ to $\udag$ as $n\to\infty$.
Hence, for the Bregman distance we get, again using \eqref{udagnuast} with $\usol=\udag$, that
\[
\begin{aligned}
D_q(\udag_n,\udag)=&
\tfrac{1}{q}\normE{\udag_n}^q-\tfrac{1}{q}\normE{\udag}^q + \dupE{J_q^{E\to E^*}(\udag),\udag-\udag_n}\\
\leq&\dupE{J_q^{E\to E^*}(\udag),\udag-\udag_n} \to 0 \mbox{ as }n\to\infty
\end{aligned}
\]
by the already shown weak convergence.
This proves the assertion in case of exact data, since then we have $u_n=\udag_n$.\\
In case of noisy data we can estimate the Bregman distance between $u_n$ and $\udag_n$ by means of Lemma \ref{stability_leasterror}:
\[
D_q^{sym}(u_n,\udag_n)^{\frac{1}{q}}\leq \hat{\kappa}_n \delta\,.
\]
So by choosing $n=n(\delta)$ such that $n(\delta)\to\infty$ and $\hat{\kappa}_{n(\delta)} \delta\to0$  as $\delta\to0$, we have
\[ 
D_q(\udag_n,\udag)\to0 \mbox{ and } D_q^{sym}(u_n,\udag_n)^{\frac{1}{q}}\to0 \mbox{ as }\delta\to0\,.
\]
However, the Bregman distance does not satisfy a triangle inequality, thus we need $q$-convexity of $E$ at this point to conclude from Lemma \ref{stability_leasterror} and \eqref{qconvex}
\[
\normE{\udag_n-\udag}\to0 \mbox{ and } \normE{u_n-\udag_n}\to0 \mbox{ as }\delta\to0\,,
\]
thus the assertion.
\end{proof}

\begin{remark} The approximation property \eqref{convQn} is ensured, e.g., by choosing $Z_n$ according to \eqref{Fspace}.

\end{remark}
\subsection{Convergence with a posteriori choice of $n$ -- the monotone error rule}

Under the conditions of Lemma \ref{characterization_leasterror} we can carry over some results for the least error method from the Hilbert space setting by closely following \cite{HAG02,GHK14}. In particular we will show monotonicity of the error measured in the Bregman distance defined by \eqref{Bregmandistance}, as well as convergence if the stopping index determined by the monotone error rule goes to infinity as $\delta\to0$, see \cite[Theorem 2]{GHK14}.
\begin{theorem}\label{conv_leastserror_me}
Let the assumptions of Lemmas \ref{welldef_leasterror} (i), (ii) and \ref{characterization_leasterror} be satisfied.  Then for $u_n$ defined by the least error method we have 
\begin{enumerate}
\item[(a)] There exists $v_n\in Z_n $ such that $u_n=(J_q^{E\to E^*})^{-1}(A^*v_n)$.
\item[(b)] With $v_n$ as in (a), the identity $\normE{u_n}^q=\dupF{v_n,\fdel}$ holds. If $Z_n\subseteq Z_{n+1}$ for all $n\in \mathbb{N}$, then
\[
\normE{u_n}\leq \normE{u_{n+1}} \mbox{ for all } n\in \mathbb{N}.
\]
\item[(c)] With $d_{ME}(n)$ defined by  
\[
d_{ME}(n)= \frac{\dupF{v_{n+1}-v_n,\fdel}}{q\normFd{v_{n+1}-v_n}}
\]
the identities
\[
d_{ME}(n)= \frac{\normE{u_{n+1}}^q-\normE{u_{n}}^q}{q\normFd{v_{n+1}-v_n}}=\frac{D_q(u_{n+1},u_n)}{\normFd{v_{n+1}-v_n}}
\]
and the estimate 
\[
D_q(\udag,u_{n+1})-D_q(\udag,u_n) \leq -(d_{ME}(n)-\delta) \normFd{v_{n+1}-v_n}
\]
hold.
In particular, if
\[
\forall n\in\N\,:\ Z_n\subseteq Z_{n+1},
\] 
then by minimality of $u_n$, we have $d_{ME}(n)\geq0$ and the error measured in the Bregman distance is monotonically decreasing as long as 
\be{deltadME}
\delta \leq d_{ME}(n)\,.
\ee
\item[(d)] Let $n=n_{ME}(\delta)$ be the first index such that \eqref{deltadME} is violated.

If $n_{ME}(\delta)\to\infty$ as $\delta\to0$ and \eqref{convQn} holds, then $\normE{u_{n_{ME}(\delta)}-\udag}\to0$ as $\delta\to0$ provided that $E$ is smooth and $q$-convex.
\end{enumerate}
\end{theorem}
\begin{proof}
Item (a) has already been proven in Lemma \ref{characterization_leasterror}.\\
Since the duality mapping satisfies 
\be{id_dualitymapping}
\dupE{J_q^{E\to E^*}(w),w}=\normE{w}^q\,,
\ee
we get the first part of item (b):
\[
\begin{aligned}
\normE{u_n}^q=&\dupE{J_q^{E\to E^*}(u_n),u_n}=\dupE{A^*v_n,u_n}\\
=&\dupF{v_n,Au_n}=\dupF{v_n,\fdel}\,.
\end{aligned}
\]
Note that $v_n$ as in (a) satisfies $\dupF{v_n,Au_n}=\dupF{v_n,\fdel}$ and $\dupF{v_n,Au_{n+1}}=\dupF{v_n,\fdel}$, due to the assumption $Z_n\subseteq Z_{n+1}$. Then \eqref{leasterror} 
yields the second part of (b).

The first identity in (c) is an immediate consequence of (b), while the second one follows from $\dupF{v_n,Au_{n+1}}=\dupF{v_n,Au_{n}}$ which can be rewritten as $\dupE{J_q^{E\to E^*}(u_n),u_{n+1}-u_n}=0$.

 Considering the differences between the Bregman distances and using that the term $\tfrac{1}{q}\normE{\usol}^q$ cancels out we get
\[
\begin{aligned}
&D_q(\usol,u_{n+1})-D_q(\usol,u_n)\\
=&\tfrac{1}{q}\normE{u_n}^q-\tfrac{1}{q}\normE{u_{n+1}}^q\\
&\quad 
+ \dupE{J_q^{E\to E^*}(u_{n+1}),u_{n+1}-\usol}-\dupE{J_q^{E\to E^*}(u_n),u_n-\usol}\\
=&\tfrac{1}{q}\normE{u_n}^q-\tfrac{1}{q}\normE{u_{n+1}}^q 
+ \dupE{J_q^{E\to E^*}(u_{n+1}),u_{n+1}}-\dupE{J_q^{E\to E^*}(u_n),u_n}\\
&\quad-\dupE{J_q^{E\to E^*}(u_{n+1})-J_q^{E\to E^*}(u_n),\usol}\\
=&\tfrac{1}{q^*} \normE{u_{n+1}}^q - \tfrac{1}{q^*} \normE{u_n}^q-\dupE{A^*(v_{n+1}-v_n),\usol}\\
=&\tfrac{1}{q^*} \dupF{v_{n+1}-v_n,\fdel}-\dupF{v_{n+1}-v_n,A\usol}\\
=&-\tfrac{1}{q}\dupF{v_{n+1}-v_n,\fdel}+\dupF{v_{n+1}-v_n,\fdel-f}\\
\leq&-\tfrac{1}{q}\dupF{v_{n+1}-v_n,\fdel}+\normFd{v_{n+1}-v_n}\delta\\
=&-(d_{ME}(n)-\delta)\normFd{v_{n+1}-v_n}\,,
\end{aligned}
\]
where we have  used again \eqref{id_dualitymapping} in the second equality.

Let $n_0(\delta)$ be an a priori stopping rule satisfying \eqref{aprioristoppingrule1}, let $(\delta^k)_{k\in\N}$ be a sequence of noise levels tending to zero and denote by $n_0^k=n_0(\delta^k)$, $n_{ME}^k=n_{ME}(\delta^k)$ the stopping indices chosen by the a priori and the monotone error rule, respectively.\\
If there exists $k_0$ such that $n_{ME}^k>n_0^k$ for all $k\geq k_0$, then by monotone decay of the error up to $n_{ME}^k$ we have
$D_q(\usol,u_{n_{ME}^k})\leq D_q(\usol,u_{n_0^k})\to0$ as $k\to\infty$.\\
Otherwise there exists a subsequence $(k_l)_{l\in\N}$ such that for all $l\in\N$ we have $n_{ME}^{k_l}\leq n_0^{k_l}$ and therefore $\hat\kappa_{n_{ME}^{k_l}}\leq\hat\kappa_{n_0^{k_l}}$, so the right hand limit in \eqref{aprioristoppingrule1} together with  Lemma \ref{stability_leasterror} implies $\normE{\udag_{n_{ME}^{k_l}}-
u_{n_{ME}^{k_l}}}\stackrel{l\to\infty}{\to}0$. On the other hand, by assumption we have $n_{ME}^{k_l}\stackrel{l\to\infty}{\to}\infty$, thus by Theorem \ref{conv_leasterror_apriori}, $\normE{\udag_{n_{ME}^{k_l}}-\udag}\to0$ as $l\to\infty$.
Thus a subsequence-subsequence argument yields the assertion.
\end{proof}
\begin{remark} Convergence in the degenerate case when $(n_{ME}(\delta))$ has finite accumulation points remains an open problem even in Hilbert spaces.

As regards a relation of the type
\[
d_{ME}(n)\leq \normF{Au_n-\fdel}/2,\,\,\,\forall n\in\mathbb{N},
\]
 shown in Hilbert spaces (see, e.g., \cite[Th.2, 5)]{GHK14},  it is not clear whether such a connection could be established in the Banach space framework.
\end{remark}
\section{On the requirements for spaces and subspaces}\label{sec_space}

The three projection methods investigated in this work require different theoretical settings as concerns stability and convergence. 

Note that reflexivity of the space $E$ is essential in convergence results for the least error method, thus ruling out the case $E=C(\overline{\Omega})$ or $E=\mathcal{M}(\Omega)$, while allowing $F=C(\overline{\Omega})$ (thus , e.g.,  collocation) or $F=\mathcal{M}(\Omega)$ (for modelling impulsive noise). 

The additional restrictions on uniform boundedness (e.g., \eqref{bd}) will be discussed in the following section; 
they are more severe in case of a posteriori choice of $n$, a fact which is already known from the Hilbert space setting.



The preimage and image space combinations we are interested in are 
\be{LpLr}
E=L^p(\Omega)\,, \quad F=L^r(\Omega)\,, \quad p,r\in(1,\infty)\,,
\ee
\be{LpC}
E=L^p(\Omega)\,, \quad F=C(\overline\Omega)\,, \quad p\in(1,\infty)\,,
\ee
\be{MLr}
E=C(\overline\Omega)^*=\mathcal{M}(\Omega)\,, \quad F=L^r(\Omega)\,, \quad p\in(1,\infty)\,,
\ee
\be{MC}
E=\mathcal{M}(\Omega)\,, \quad F=C(\overline\Omega)\,, 
\ee
for some smooth open domain $\Omega\subseteq\R^d$.
$L^p$ spaces with $p\in(1,\infty)$ are reflexive, smooth and $q(p)$-convex with $q(p)=\max\{2,p\}$, the duality mappings, which are given by
\[
(J_{q(p)}^{L^p\to L^{p^*}}(w))(x)=\norm{w}_{L^p}^{q(p)-p}|w(x)|^{p-1}\mbox{sign}(w)(x) 
\quad p^*=\frac{p}{p-1}
\]
are invertible with 
\[
(J_{q(p)}^{L^p\to L^{p^*}})^{-1}=J_{{q(p)}^*}^{L^p\to L^{p^*}}
\quad q(p)^*=\frac{q(p)}{q(p)-1}
\] 
(see, e.g., \cite[Section II.2]{SKHK12}), and if $r\geq2$, i.e., $q(r)=\max\{2,r\}=r$, then $J_{q(r)}^{L^r\to L^{r^*}}$ is additionally Gateaux differentiable with Gateaux derivative
\[
(J_{q(r)}^{L^r\to L^{r^*}})'(g))[h](x)=(r-1)|g(x)|^{r-2}h(x)\,.
\]
Therefore in case \eqref{LpLr} all well-definedness, characterization, stability and convergence results 
Lemmas 
\ref{welldef_projmeth},
\ref{welldef_leastsquares},
\ref{welldef_leasterror},
\ref{characterization_leastsquares},
\ref{characterization_leasterror},
\ref{stability_projmeth},
\ref{stability_leastsquares},
\ref{stability_leasterror},
Corollary 
\ref{conv_leastsquares_apriori},
and Theorems 
\ref{conv_projmeth_apriori},
\ref{conv_leasterror_apriori},
\ref{conv_projmeth_dp},
\ref{conv_leastsquares_dp},
\ref{conv_leastserror_me},
are applicable.
In case \eqref{LpC}, we still have all these results except for those on stability of the least squares method, Lemma \ref{stability_leastsquares} unless the projection spaces are chosen appropriately (cf. Remark \ref{Qpcontinuous}).
Likewise, in case \eqref{MLr} all results except for those concerning the least error method apply.
Finally, in the situation \eqref{MC}, only the results for the general projection method, 
Lemmas 
\ref{welldef_projmeth},
\ref{stability_projmeth},
and Theorems 
\ref{conv_projmeth_apriori},
\ref{conv_projmeth_dp},
remain valid.

\section{Applications}\label{sec_applic}

We will now consider applicability of the results derived in the previous sections for concrete discretizations, so that the crucial conditions for convergence and stability 
\eqref{convPn}, \eqref{bd_conv}, \eqref{bd1_conv}, \eqref{bdDP_conv}, \eqref{kappa*appr}, \eqref{kappa2*appr0_conv}, \eqref{convQn},
will become conditions on the smoothing properties of the forward operator. These will be interpreted for the case of integral equations. For certain test examples we will also provide numerical experiments.

\subsection{On convergence conditions for projection methods}

For applying the results from Sections \ref{sec_genproj}, \ref{sec_leastsquares}, \ref{sec_leasterror}, in the respective cases, it still remains to verify the crucial convergence conditions.
However, the convergence conditions \eqref{bd_conv}, \eqref{bd1_conv}, \eqref{bdDP_conv}, \eqref{kappa*appr}, \eqref{kappa2*appr0_conv} (recall the corrsponding sufficient boundedness conditions \eqref{bd}, \eqref{bd1}, \eqref{bdDP}, \eqref{kappa*appr0}, \eqref{kappa2*appr0})  require an appropriate trade-off between stability and approximation. Note that these conditions are only needed for the general projection and the least squares method, but not for the least error method.

We will now illustrate these conditions for integral equations with discretization in spline spaces.

Let $k,n \in \N$, $h=\frac1n$, $1<p<\infty$, $1<r<\infty$. We denote by $S^{(l)}_{k-1}(I_h)$ the spline space defined as the set of functions $w_h \in C^l[0,1]$, which in each subinterval $I_h^i:[(i-1)h,ih], i=1,\dots,n$ are polynomials of order $\leq k-1$: $w_h|_{I_h^i} \in \Pi_{k-1}$. 
The case of potentially discontinuous piecewise polynomial functions $w_h$ will be denoted by 
$S^{(-1)}_{k-1}(I_h)$.\\
We recall below  several well-known properties of splines.
\begin{enumerate}
\item[1)] Approximation property: 
\[
\begin{aligned}
\forall v \in W^{l,p}(0,1), \,\,
 \exists v_h \in S^{(-1)}_{k-1}(I_h): \,\,
 \|v-v_h\|_{L^{p}(0,1)} \leq C_{\text{app}} h^{\min(k,l)} \|v\|_{W^{l,p}(0,1)},\\
\forall v \in C^l[0,1], \,\,\,
 \exists v_h \in S^{(-1)}_{k-1}(I_h): \,\,\, 
 \|v-v_h\|_{C[0,1]} \leq C_{\text{app}} h^{\min(k,l)} \|D^lv\|_{C[0,1]},
\end{aligned}
\]
where $D^l$ is the differential operator of order $l$.
\item[2)] Stability property:
\[\forall v_h \in S^{(l-1)}_{k-1}(I_h), \quad
\|D^lv_h\|_{L^p(0,1)} \leq C_{\text{inv}} n^l \|v_h\|_{L^r(0,1)}
\leq C_{\text{inv}} n^l \|v_h\|_{C[0,1]}.\]
\end{enumerate}
On each subinterval $I_h^i$ we define the local projection $P_n$, using an $L^2(I_h^i)$-orthogonal basis $\{\phi^{I_h^i}_1,\ldots,\phi^{I_h^i}_k\}$ of $\Pi_{k-1}$:
\[
(P_n w)(t) =\sum_{j=1}^k \int_{I_h^i}\phi^{I_h^i}_j(s) w(s) \, ds \, \phi^{I_h^i}_j(t) \, \quad t\in I_h^i.
\]
We consider $P_n$ as a mapping $P_n:L^p(0,1)\to L^p(0,1)$, with range $\mathcal{R}(P_n)=E_n$, 
\[
E_n=\{v\in L^p(0,1)\, \vert \, \forall i=1,\dots,n\, : \, v\vert_{I_h^i}\in \Pi_{k-1}\}\,.
\]
By the $L^2(I_h^i)$ orthogonality of the basis functions, it is easily checked that $P_n^*$ is defined in exactly the same manner, but considered as a mapping $L^{p^*}(0,1)\to L^{p^*}(0,1)$, again with range $\mathcal{R}(P_n^*)=E_n$.
Obviously $I-P_n^*$ annihilates polynomials of degree lower or equal to $k-1$ on each $I_h^i$.

For checking conditions \eqref{bd}, \eqref{bd1}, \eqref{bdDP}, \eqref{kappa*appr0}, \eqref{kappa2*appr0} we can use the following lemma which follows from the approximation property of splines.  
\begin{lemma}\label{appr-error}
Let $A\in L(E,F)$, $E=L^p(0,1)$, $E_n=S^{(-1)}_{k-1}(I_h)$. If  
$F=L^r(0,1)$, $A^*Z_n \subset  W^{l,p^*}(0,1)$ or $F=C[0,1]$, $A^*Z_n \subset C^l[0,1]$, then
\[
\begin{aligned}
&\sup_{w\in E,\normE{w}=1} \sup_{z_n\in Z_n ,\normFd{z_n}=1}\dupF{z_n,A(I-P_n)w}\\
&=\sup_{w\in E,\normE{w}=1} \sup_{z_n\in Z_n ,\normFd{z_n}=1}\dupF{(I-P_n^*)A^*z_n,w} \leq C_{\text{app}} h^{\min(k,l).}
\end{aligned}
\]
\end{lemma}

Due to this lemma, for conditions \eqref{bd}, \eqref{bd1}, \eqref{bdDP}, \eqref{kappa*appr0}, \eqref{kappa2*appr0}, we need the inequality $k\geq l$ and the estimate $\kappa_n\leq C n^l$. We are able to guarantee the latter estimate only for specific operators.

\begin{lemma}\label{kappa-estimate}
Let $A\in L(E,F)$ with $E=L^p(0,1)$, $E_n=S^{(-1)}_{k-1}(I_h)$ and $F=L^r(0,1)$ or $F=C[0,1]$. If for all $w_n \in S^{(-1)}_{k-1}(I_h)$ we have
$\|w_n\|_E\leq C_1\|D^lAw_n\|_E$ and $v_n:=Aw_n \in S^{(l-1)}_{k+l-1}(I_h)$, then 
\be{kappa-est}
\kappa_n=\sup_{w_n\in S^{(-1)}_{k-1}(I_h) }\frac{\|w_n\|_{L^p(0,1)}}{\|Aw_n\|_F} \leq C n^l, \quad C=C_1 C_{\text{inv}}.
\ee
\end{lemma}
\begin{proof}
Let $w_n \in S^{(-1)}_{k-1}(I_h)$ satisfy the above assumptions. Since $v_n:=Aw_n$ is a spline of order $k+l-1$ with increased global smoothness, one can  apply the stability property of splines to $v_n$:   
$\|w_n\|_E\leq C_1\|D^l Aw_n\|_E \leq  C_1 C_{\text{inv}} n^l \|Aw_n\|_F.$
\end{proof}

The conditions of Lemmas \ref{appr-error} and \ref{kappa-estimate} are satisfied for integral equations of the first kind
\be{int-eq}
(A u)(t):= \int_0^1 K(t,s) u(s)\,ds =f(t),\ t\in [0,1], 
\ee
whose kernels are Green's functions for the differential operator $D^l, l \in \N$ under different homogeneous boundary conditions, such that the equation $D^lz=0$  has only the trivial solution $z=0$. Here $K(t,s)$ has different forms $K_1(t,s)$ and $K_2(t,s)$  for regions $0\leq s < t\leq 1$ and $0\leq t \leq  s\leq 1$ respectively. 
Note that a Green's function of $D^l$ with boundary conditions $f^{(j)}(0)=0, j=0,1,\dots,l-1$  is given by the Volterra kernel $K_1(t,s)=(t-s)^{l-1}/(l-1)!$, $K_2(t,s)=0$. For $l=2$ and boundary conditions $f(0)=f(1)=0$ we have $K_1(t,s)=s(t-1)=K_2(s,t)$, for $l=4$ and boundary conditions $f(0)=f'(0)=f(1)=f'(1)=0$ we have $K_1(t,s)=-s^2(1-t)^2(s+2st-3t)/6=K_2(s,t)$. 
Let us formulate the convergence theorem. 
\begin{theorem}\label{convGP-LS}
Consider $A\in 
L(E,F)$ defined by (\ref{int-eq}) with  $E=L^p(0,1)$, $1<p<\infty$, where 
\[
\begin{aligned}
&F=L^r(0,1), 1<r<\infty\mbox{ and }f \in W^{l,r}(0,1)\mbox{ or }\\
&F=C[0,1]\mbox{ and }f \in C^l[0,1]
\end{aligned} 
\]
is assumed. Let $K(t,s)$ be a Green's function of $D^l$ with homogeneous boundary conditions such that $D^lz=0$  has only the trivial solution $z=0$, let $f(t)$ satisfy these boundary conditions.\\ 
Then the following statements hold:

(i) Equation (\ref{int-eq}) has a unique solution $\usol$. 

(ii) Let  $E_n=S^{(-1)}_{k-1}(I_h)$ with $k\geq l$.
Then the least squares method determines a unique approximation $u_n \in E_n$ for all $n \in \N$.

(iii) If \eqref{nsp} hold the general projection method (\ref{projmeth}) determines a unique approximation $u_n \in E_n$  for all $n \in \N$.

Under these assumptions we have for both methods convergence $\|u_n-\usol\| \to 0$ as $n \to \infty$ in case of exact data $\delta=0$. In case of noisy data one has convergence $\|u_{n(\delta)}-\usol\| \to 0$ as $\delta  \to 0$, if $n=n(\delta)$ is chosen a priori such that $n(\delta) \to \infty$, $n(\delta)^l \delta \to 0  \text{ as } \delta  \to 0$ or a posteriori according to the discrepancy principle, where in the least squares method $b>1$, while in the general projection method 
assumptions \eqref{tau} and $b>\tau+1$ are assumed.   
\end{theorem}
\begin{proof}
The assumptions of Lemmas \ref{appr-error} and \ref{kappa-estimate} are satisfied, since for any $w \in E$ we have $\|w\|_E=\|D^lAv\|_E$ and for any $w_n \in E_n$ we have $Aw_n \in E_n$ with increased power and global smoothness of the spline.
The assertions follow with Lemmas \ref{appr-error}, \ref{kappa-estimate} 
from Theorems \ref{conv_leastsquares_apriori_gen}, \ref{conv_leastsquares_dp} for the least squares method, and
from Theorems \ref{conv_projmeth_apriori}, \ref{conv_projmeth_dp}, and inequalities 
\eqref{kappa_tau_n}
\eqref{tau} 
for the general projection method, respectively.
\end{proof}

One can compare the above results to their counterparts in Hilbert spaces (see \cite{VH85}).
  For the general projection method (\ref{projmeth}), the Hilbert space analog of Theorems \ref{conv_projmeth_apriori}, \ref{conv_projmeth_dp} is the following.
 \begin{theorem}\label{Hilbert-GP}
Let $A\in L(E,F)$, where $E$ and $F$ are Hilbert spaces. Let $f \in \mathcal{R}(A)$ and $P_n:E \to E_n$, $Q_n:F \to F_n$, $Q_n':F \to AE_n$ be orthoprojectors, where $F_n$ are finite dimensional subspaces of $F$. Let  
the following conditions (i)-(iii) hold:
\begin{enumerate}
\item[(i)] $\forall u\in E \, : \ \normE{P_n u-u}\to0\mbox{ as }n\to\infty$,
\item[(ii)] $\forall n\in \N\, : \ \mathcal{N}(A^*)\cap F_n=\{0\}$,
\item[(iii)] $\exists \tau^* <\infty\, \forall z_n\in F_n\, : \ \tau^* \normE{P_n A^* z_n}\geq\normE{A^* z_n}$. 
\end{enumerate}

Then equations $Au=f$ and (\ref{projmethQn}) have unique solutions $\usol\in E$ and $u_n\in E_n$ respectively. If $\delta=0$, then  $\normE{u_{n}-\usol}\to0$ as $n\to \infty$ ((i)-(iii) are  necessary and sufficient conditions of this convergence for arbitrary $f \in \mathcal{R}(A)$).
If $\delta>0$, then for an a priori choice of $n=n(\delta)$ such that
$n(\delta)\to\infty, \quad \delta\cdot\kappa_{n(\delta)}^*\to0 \quad (\delta \to 0)$
one has $\normE{u_{n(\delta)}-\usol}\to0$ as $\delta\to0$. If $\delta>0$ and the following additional conditions (iv)-(vi) hold  
\begin{enumerate}
\item[(iv)] $ \mathcal{N}(A)\cap E_n=\{0\}$, 
\item[(v)] $\exists \tau<\infty \ \forall v_n\in E_n \, : \ \tau\normF{Q_n Av_n}\geq \normF{A v_n}$,
\item[(vi)] $\exists C<\infty \ \forall n\in\N \, : \ \kappa_{n+1}\normF{(I-Q'_n)A}\leq C$, 
\end{enumerate}
then convergence $\normE{u_{n(\delta)}-\usol}\to0$ as $\delta\to0$ also holds for a choice of $n=n(\delta)$ by the discrepancy principle with $b>\tau$.
\end{theorem}

Note that in Hilbert spaces, conditions (iii), (v) are automatically fulfilled by the least error method $(E_n=A^*F_n)$  and by the least squares method $(F_n=AE_n)$  respectively, and that for condition (vi) the inequality $\|(I-Q'_n)A\|\leq\|(I-P_n)(A^*A)^\frac{1}{2l}\|^l, \forall l \in \N$ is useful.
Conditions (iii), (v) here seem to be weaker than the corresponding conditions (\ref{bd}), (\ref{bd1}), (\ref{bdDP}) in the Banach space theorems.

For the least squares method (\ref{leastsquares}), the Hilbert space analog of Theorems \ref{conv_leastsquares_apriori_gen}, \ref{conv_leastsquares_dp} is Theorem \ref{Hilbert-LS-gen} and the analog of Theorem \ref{convGP-LS} is Theorem \ref{Hilbert-LS-Green}.

\begin{theorem}\label{Hilbert-LS-gen}
Let $A\in L(E,F)$, where $E$, $F$ - Hilbert spaces, $\mathcal{N}(A)=\{0\}$, $f \in \mathcal{R}(A)$, $P_n:E \to E_n$ orthoprojector,  let $\|P_n u-u\|\to0$ as $n\to\infty$ for all $u\in E$, and let
\be{bdLS-H}
 \exists l \in \N \quad \exists C<\infty: \quad(\kappa_{n} +\kappa_{n+1})\normF{(I-P_n)(A^*A)^{\frac{1}{2l}}}^l \leq C\quad \forall n \in \N.
\ee
 Then equations $Au=f$ and (\ref{leastsquares}) have unique solutions 
  $\usol\in E$ and $u_n\in E_n$ respectively. If $\delta=0$, then  $\|u_{n}-\usol\|\to0$ as $n\to \infty$. 
If $\delta>0$, then $\|u_{n(\delta)}-\usol\|\to0$ as $\delta\to0$ for an a priori choice of $n=n(\delta)$ such that
$n(\delta)\to\infty, \quad \delta\cdot\kappa_{n(\delta)}\to0$ as $\delta \to 0$
and also for a choice of $n=n(\delta)$ according to the discrepancy principle with $b>1$.
\end{theorem}

\begin{theorem}\label{Hilbert-LS-Green}
Let $E=F=L^2(0,1)$, $K(s,t)$ in (\ref{int-eq}) be a Green's function of the differential operator
\[
L_lz=\sum_{j=0}^l b_j(t)z^{(j)}, \quad b_j \in C[0,1],\quad b_m(t) \not=0 \quad\forall t \in(0,1)
\]
with boundary conditions 
$\sum_{j=0}^{l-1} \alpha_{i,j}z^{(j)}(0)+\beta_{i,j}z^{(j)}(1)=0$, $(i=1,\dots,l)$ 
such that $L_lz=0$ only has the trivial solution $z=0$, and let
$f(t)$ satisfy these boundary conditions. Then 
equation (\ref{int-eq}) has a unique solution $\usol$ and the least squares method with $E_n=S^{(-1)}_{k-1}(I_h)$  determines
a unique approximation $u_n \in E_n \quad \forall n, k \in \N$. 
Convergence $\normE{u_{n(\delta)}-\usol}\to0$ as $\delta\to0$ holds with an a priori choice of $n=n(\delta)$ such that
$n(\delta)\to\infty, \quad \delta\cdot{n(\delta)}^l\to0 \quad (\delta \to 0)$
and also with a choice of $n=n(\delta)$ by the discrepancy principle with $b>1$.
\end{theorem}
In Theorems \ref{conv_leastsquares_apriori_gen}, \ref{conv_leastsquares_dp} we needed instead of condition (\ref{bdLS-H}) the conditions (\ref{kappa*appr0}), (\ref{kappa2*appr0}) corresponding to the special case $l=1$ in (\ref{bdLS-H}).
If $\mathcal{R}(A^*)=\mathcal{R}(A^*A)^{\frac12}\subset W^{l,2}$, then 
$\mathcal{R}((A^*A)^{\frac{1}{2l}}) \subset W^{1,2}$ and in case 
$E_n=S^{(-1)}_{k-1}$,  
$\kappa_n \leq Cn^l$ condition (\ref{bdLS-H}) is satisfied for all $k \in \N$, but (\ref{kappa*appr0}), (\ref{kappa2*appr0}) require $k \geq l$ in Theorem \ref{convGP-LS}.\\ 
We list below several open problems:
\begin{enumerate}
\item[(1)] Is it possible to weaken the assumption $k \geq l$?
\item[(2)] Is it possible to extend the results of Theorem \ref{convGP-LS} using  a more general operator $L_l$ instead of the operator $D^l$, as in Theorem \ref{Hilbert-LS-Green}?
\end{enumerate}
Concerning (1), computational results for the collocation method indicate that $k=l$ is really needed there.
Note that (2) can be reduced to the (also open) question, whether the following lemma, proved in \cite{VH85} for the case $q=r=2$, remains valid for general $q,r\in[1,\infty]$. 

\begin{lemma}
Let 
$B \in L(L^q,L^r), \quad  W_0^{l,r}(0,1) \subset B(L^q(0,1)) \subset W^{l,r}(0,1)$,
where $W_0^{l,r}(0,1)=\{z\in W^{l,r}(0,1),z^{(j)}(0)=z^{(j)}(1)=0,j=0,\dots,l-1\}$, $L^q=L^q(0,1)$, $1<q<\infty$, $1<r<\infty$.
Then 
$\|B^*v\|_{L^{q^*}} \geq C_1 \|D^{(-l )}v\|_{L^{r^*}}$, $\forall v \in L^{r^*}$, $q^*=q/(q-1)$, $r^*=r/(r-1)$,    
where $D^{(-l )}v=D^l\Gamma_lv$, $\Gamma_l:L^{r^*} \to W_0^{l,r^*}$ is the inverse to the differential operator $D^{2l}$ for the boundary conditions $z^{(j)}(0)=z^{(j)}(1)=0, j=0,1,\dots,l-1$.
\end{lemma}

In the next section we consider the collocation method as a special case of the general projection method, applying Theorem \ref{convGP-LS} to a Volterra integral equation of the first kind and estimating $\tau$. Note that in \cite{HAG02} a collocation method for integral equations of the first kind is considered using kernel functions for basis functions, the number of which was determined by the monotone error rule.

\subsection{On the collocation method for a Volterra integral equation}

We consider a Volterra integral equation of the first kind 
\be{Volterra}
(A u)(t):= \int_0^t K(t,s) u(s)\,ds =f(t),\ t\in [0,1] 
\ee
with the operator
$ A\in L(L^p(0,1),C[0,1]),\  1\leq p\leq \infty$. A special case of equation (\ref{Volterra}) is the model problem 
\be{Volt-conv}
(A u)(t):= \int_0^t(t-s)^{l-1} u(s)\,ds =f(t),\ t\in [0,1].
\ee 
In the collocation method we find $u_n \in E_n=S^{(-1)}_{k-1}(I_h)$ such that
\[
Au_n(t_{i,j})=\fdel(t_{i,j}), \quad i=1,...,n, \ j=1,...,k
\]
where $t_{i,j}= (i-1+c_j)h\in [0,1]$, $i=1,...,n$, $j=1,...,k$ are collocation nodes and  $0<c_1<...<c_k\leq 1$ are collocation parameters whose choice is essential.

In \cite{Brunner},  spline collocation is considered in case $\delta=0$, $E=L^\infty$, $F=C$,  $|K(t, t)| > 0$ (case $l=1$ in \eqref{Volt-conv}). Theor. 2.4.2  in \cite{Brunner} (page 123) proves that convergence holds if and only if 
\[\prod_{j=1}^{k} (1-c_j)/c_j <1 \,.\]
In \cite{Eggerm83} the case $K(t, t) = 0$
is considered, where $\frac{\partial K(t,s)}{\partial t}|_{t=s}\not=0$
(case $l=2$ in (\ref{Volt-conv})) and convergence if  $c_k=1$ and $\prod_{j=1}^{k-1} (1-c_j)/c_j <1$ is proven. Convergence of the collocation method for equation \eqref{Volt-conv} in case $l>2$ seems to be an open problem.

In our numerical experiments below we will use the discrepancy principle for the choice of a  proper number $n=n(\delta)$ of the subintervals, thus we use the first $n$ such that  $\normF{Au_n-f^{\delta}} \leq b \delta$. According to Theorem \ref{conv_projmeth_dp} we need that $b>\tau +1$, so the value of $\tau$ in \eqref{tau} is needed. For the use of 
inequality $\tau_n\leq \tau$ for all $n\in\N$  
we need to estimate
\[ \tau_n = \sup_{w_n\in E_n} \frac{\| Aw_n \|}{\sup_{z_n\in Z_n ,\normFd{z_n}=1}\dupF{z_n,Aw_n}} 
= \sup_{w_n\in E_n} \frac{\sup_{t\in [0,1]} |Aw_n(t)|}{\sup_{i,j} |Aw_n(t_{i,j})|}.
\] 
In the numerical experiments of the next section we solve equation \eqref{Volt-conv} with $l=2$. We use linear splines $k=2$ and collocation nodes $t_{i1}=(i-1)h+ch$ with $c \in (0.5,1)$ and $t_{i2}=ih$.
It can be shown that $\tau=\tau(c)$  depends on $c$ in the form 
\be{tau(c)}
\tau(c) =1+\frac{4(y^2-y+1)^{3/2}-4y^3+6y^2+6y-4}{27y^2(2c-1)(1-c)},\quad y=c(-2c^3+c^2+1).
\ee 
Actually, it is sufficient to consider cubic functions $z(t)$ on the interval $[0,1]$ 
which satisfy $z(0)=z(c)=1$, $z(1)=-1, z'(0)=2/(c(1-c)(2c-1))$. 
The last equality is the bound on the derivative of the cubic spline $Aw_n$ at the points $ih$ under conditions $|Aw_n(t_{i,j})|\leq 1$ if $n \to \infty$.
The value of $\tau(c)$ in (\ref{tau(c)}) is the maximum of $z(t)$. 

\subsection{Numerical example}

We consider equation (\ref{Volt-conv}) with the exact solutions $\usol(s)=s^r, r \in \{1/2, 3/2\}$, where the exact right hand side is computed as $f(t)= (Au)(t)$. 
The noisy data were generated by the formula $f_{\delta}(t_{i,j})=f(t_{i,j})+\delta \theta_{i,j}$, where $\delta=10^{-m}, m \in \{2,...,7\}$ and $\theta_{i,j}$ are random numbers with normal distribution, normed after being generated:
 $\max_{i,j} |\theta_{i,j}|=1.$ In the space setting we used $p=1$, i.e., we consider $A$ as an operator from $L^1(0,1)$ to $C[0,1]$.

In our numerical experiment we took $k=2$ (linear splines) and used collocation nodes        $t_{i1}=(i-1)h+ch$ with $c \in (0.5,1)$ and $t_{i2}=ih$.
 Table \ref{tab:numerics_apost} contains the results for
$c \in \{0.6,\, 0.7,\, 0.8, 0.9\}$; according to  formula \eqref{tau(c)}, the corresponding values of $\tau(c)$ are 5.67,  4.10,  4.22 and 6.51 respectively. For fulfilling the theoretical requirement \eqref{tau} in Theorem \ref{conv_projmeth_dp} we actually used $b(c)=1.01+\tau(c)$ in the discrepancy principle. The discrepancy principle gave a number $n_D$ of subintervals with corresponding error $e_D=\|u_{n_D} -\usol\|$.
We also found the optimal number $n_{opt}$ of subintervals and the corresponding error 
$e_{opt}=\min_{n \in \N} \normE{u_n -\usol}= \normE{u_{n_{opt}} -\usol}$, as well as the best coefficient $b=b_{opt}$ for the choice of $n=n(\delta)$ in the discrepancy principle according to $b_{opt}=\normF{Au_{n_{opt}}-f_\delta}/\delta$. 

Table \ref{tab:numerics_apost} contains our results for the exact solutions $\usol(s)=s^r$ with $r=1/2$ (left) and $r=3/2$ (right). Columns $r_b$ and $r_e$ contain the ratios of the $b$-values $r_b=b(c)/b_{opt}$ and the corresponding errors $r_e=e_D/e_{opt}$. 
The performance of the discrepancy principle is determined by the constant $b$.
According to column $r_b$, the lowest values of constants $b=b(c)$, needed by the assumptions of Theorem \ref{conv_projmeth_dp}, are typically 1.5 to 3 times larger than the optimal values $b_{opt}$. Nevertheless, column $r_e$ shows that the errors $e_D$ of the approximate solutions with choice of the dimension by the discrepancy principle were typically not larger than 1 to 1.4 times the optimal errors $e_{opt}$.  Comparison of the errors $e_D$ for different $c$-values suggests to use medium $c$-values 0.7 or 0.8.  
\begin{table}\caption{Results for optimal $n$ and for $n$ according to the discrepancy principle}
\label{tab:numerics_apost}
\medskip
\begin{tabular}{ |c|c||c|c|c|c|c||c|c|c|c|c|c|} 
 \hline
&&
\multicolumn{5}{|c||}{$\usol(s)=s^{1/2}$}&
\multicolumn{5}{|c|}{$\usol(s)=s^{3/2}$}
\\
\hline
 c & $\delta$ &  $n_{opt}$ & $n_D$ & $r_b$ & $e_D$  & $r_e$ &  $n_{opt}$ & $n_D$ & $r_b$ & $e_D$  & $r_e$ \\
 \hline
0.6	    & 1.E-02	&	1		&	1		& 3.1	& 0.325		& 1.00		& 1     &    1 &   2.5  & 0.502  & 1.00\\
0.6		& 1.E-03	&	2		&	1		& 2.5	& 0.289		& 1.90		& 2     &    2 &   2.2  & 0.180  & 1.00\\
0.6		& 1.E-04	&	6		&	4		& 2.7	& 0.079		& 1.17		& 6     &    3 &   3.0  & 0.092  & 2.12\\
0.6		& 1.E-05	&	12		&	7		& 2.7	& 0.040		& 1.39	    &10     &    6 &   2.7  & 0.032  & 2.03\\
0.6		& 1.E-06	&	24		&	18		& 2.2	& 0.012		& 1.14		&20     &   11 &   3.3  & 0.011  & 2.46\\
0.6		& 1.E-07	&	48		&	42		& 1.4	& 0.004		& 1.02		&34     &   22 &   2.8  & 0.003  & 1.91\\
\hline
0.7		& 1.E-02	&	1		&	1	    & 2.3	& 0.336		& 1.00		& 1     &    1 &   2.2  & 0.516  & 1.00\\
0.7		& 1.E-03	&	2		&	2		& 1.8	& 0.145 	& 1.00		& 2     &    2 &   1.7  & 0.169  & 1.00\\
0.7		& 1.E-04	&	6		&	4		& 2.1	& 0.065		& 1.14		& 6     &    4 &   2.2  & 0.054  & 1.35\\
0.7		& 1.E-05	&	12		&	8		& 2.2	& 0.025		& 1.09		&10     &    6 &   2.1  & 0.022  & 1.69\\
0.7		& 1.E-06	&	24		&	20		& 1.8	& 0.008		& 1.01		&20     &   12 &   2.5  & 0.006  & 1.56\\
0.7		& 1.E-07	&	42		&	46		& 0.9	& 0.003		& 1.01		&30     &   22 &   2.0  & 0.002  & 1.54\\
\hline
0.8		& 1.E-02	&	1		&	1		& 2.0	& 0.358		& 1.00		& 1     &    1 &   2.2  & 0.534  & 1.00\\
0.8		& 1.E-03	&	2		&	2		& 1.7	& 0.148		& 1.00		& 2     &    2 &   1.6  & 0.164  & 1.00\\
0.8		& 1.E-04	&	6		&	4		& 1.9	& 0.063		& 1.00		& 6     &    4 &   2.0  & 0.050  & 1.13\\
0.8		& 1.E-05	&	8		&	8		& 1.1	& 0.023		& 1.00	    & 8     &    6 &   1.8  & 0.019  & 1.41\\
0.8		& 1.E-06	&	20		&	20		& 1.1	& 0.008		& 1.00		&15     &   11 &   2.1  & 0.006  & 1.61\\
0.8		& 1.E-07	&	38		&	50		& 0.6	& 0.003		& 1.17		&30     &   22 &   1.9  & 0.002  & 1.33\\
\hline
0.9		& 1.E-02	&	1		&	1		& 2.1	& 0.444		& 1.00		& 1     &    1 &   2.3  & 0.600  & 1.00\\
0.9		& 1.E-03	&	2		&	2		& 2.0	& 0.175		& 1.00		& 2     &    2 &   1.8  & 0.169  & 1.00\\
0.9		& 1.E-04	&	4		&	4		& 1.5	& 0.075		& 1.00		& 4     &    4 &   1.4  & 0.059  & 1.00\\
0.9		& 1.E-05	&	8		&	8		& 1.3	& 0.025		& 1.00	    & 8     &    6 &   2.0  & 0.019  & 1.21\\
0.9		& 1.E-06	&	15		&	18		& 0.7	& 0.009		& 1.00		&15     &   11 &   2.3  & 0.006  & 1.38\\
0.9		& 1.E-07	&	32		&	46		& 0.4	& 0.004		& 1.17		&26     &   20 &   1.9  & 0.002  & 1.25\\
\hline
\end{tabular}
\end{table}
\section{Conclusions and Remarks}
In this paper we have extended some results on regularization by projection in Hilbert spaces to a more general Banach space setting. Besides being applicable in case of ``nice'' reflexive Banach spaces like $L^p$ with $p\in(1,\infty)$, some of our results also give new insights concerning certain cases of nonreflexive Banach spaces like $L^\infty,L^1,C,\mathcal{M}$ which are currently of high interest for several applications.
Analytical considerations and numerical results are provided for a Volterra integral equation in one  dimension space, using a spline discretization.
 
Future work in this context will be devoted to proving convergence rates, particularly also in nonreflexive spaces, and to more general applications in higher  dimension spaces.

\section*{Acknowledgment}
The first and third author are supported 
by the Estonian Science Foundation Grant 9120 and by institutional research funding IUT20-57
of the Estonian Ministry of Education and Research. 
The second and fourth author are supported by the Karl Popper Kolleg ``Modeling-Simulation-Optimization'' funded by the Alpen-Adria-Universit\"at Klagenfurt and by the Carinthian Economic Promotion Fund (KWF).

We thank also Reimo Palm from the University of Tartu for  numerical tests and the three referees for the careful reading of the manuscript and for the valuable comments.

\end{document}